\documentclass[11pt,twoside]{article}
\usepackage{amsmath,amssymb,epsfig,color,graphicx,blkarray}
\usepackage{epstopdf}
\usepackage{tikz}
\usepackage{caption,subcaption,tikz-qtree}
\usepackage{float}
\usepackage{dsfont}
\usetikzlibrary{matrix,arrows}
\newtheorem{proposition}{Proposition}[section]
\newtheorem{theorem}[proposition]{Theorem}
\newtheorem{lemma}[proposition]{Lemma}

\newtheorem{remark}[proposition]{Remark}
\newtheorem{example}[proposition]{Example}
\newtheorem{algorithm}[proposition]{Algorithm}
\newtheorem{assumption}[proposition]{Assumption}

\newenvironment{proof}{{\noindent \bf Proof:}}{\hfill $\fbox{}$ \vspace*{5mm}}

\numberwithin{equation}{section}

\newcommand{\be}{{\bf e}}

\newcommand{\la}{\lambda}
\newcommand{\La}{\Lambda}

\newcommand{\ca}{\mathcal{A}}

\newcommand{\ci}{\mathcal{I}}

\newcommand{\cm}{\mathcal{M}}

\newcommand{\co}{\mathcal{O}}

\newcommand{\cp}{\mathcal{P}}

\newcommand{\cv}{\mathcal{V}}
\newcommand{\cw}{\mathcal{W}}

\newcommand{\cz}{\mathcal{Z}}

\newcommand{\ri}{{\rm i}}

\newcommand{\grad}{{\rm grad\,}}

\newcommand{\qf}{{\rm qf}}

\newcommand{\tr}{{\rm tr}}

\newcommand{\diag}{{\rm diag}}

\newcommand{\bd}{\boldsymbol}

\newcommand{\R}{{\mathbb R}}

\newcommand{\Rnn}{{\mathbb R}^{n\times n}}
\newcommand{\Rmn}{{\mathbb R}^{m\times n}}

\newcommand{\BE}{\begin{equation}}
\newcommand{\EE}{\end{equation}}

\begin{document}

\title{\bf Riemannian  Newton-CG Methods for Constructing  a Positive Doubly Stochastic Matrix From Spectral Data}
\author{Yang Wang\thanks{School of Mathematical Sciences, Xiamen University, Xiamen 361005, People's Republic of China  (1139209054@qq.com).}
\and Zhi Zhao\thanks{Department of Mathematics, School of Sciences, Hangzhou Dianzi University, Hangzhou 310018, People's Republic of China (zhaozhi231@163.com). The research of this author is supported by the National Natural Science Foundation of China (No. 11601112).}
\and Zheng-Jian Bai\thanks{Corresponding author. School of Mathematical Sciences and Fujian Provincial Key Laboratory on Mathematical Modeling \& High Performance Scientific Computing,  Xiamen University, Xiamen 361005, People's Republic of China (zjbai@xmu.edu.cn). The research of this author is partially supported by the National Natural Science Foundation of China (No. 11671337) and the Fundamental Research Funds for the Central Universities (No. 20720180008).} }
\date{}
\maketitle

\begin{abstract}
In this paper, we consider the inverse eigenvalue problem for the positive doubly stochastic matrices, which aims to construct  a positive doubly stochastic matrix from the prescribed realizable spectral data. By using the real Schur decomposition, the inverse  problem is written as a nonlinear matrix equation on a matrix product manifold. We propose monotone and nonmonotone Riemannian inexact Newton-CG methods for solving the nonlinear matrix equation. The global and quadratic convergence of the proposed methods is established  under some assumptions.  We also provide invariant subspaces of the constructed solution to  the inverse  problem based on the computed real Schur decomposition. Finally, we report some numerical tests, including an application in digraph, to  illustrate the effectiveness of the proposed methods.
\end{abstract}

\vspace{3mm}

\vspace{3mm}

\section{Introduction}
The inverse eigenvalue problem (IEP) arises in many applications including structural dynamics \cite{FM95}, vibration \cite{CEH12,G04},  control design \cite{D03}, inverse Sturm--Liouville problem \cite{CC97}, and  graphs \cite{BB20,BN14}, etc. One may refer to \cite{C98,CG02,CG05,X98} and references therein for the theoretical results, computational approaches, and applications of a general IEP.

In this paper, we consider the following IEP for positive doubly stochastic matrices.

{\bf PDStIEP.} {\em Given a realizable list of $n$ complex numbers $\{\lambda^*_1, \lambda^*_2, \ldots, \lambda^*_n\}$,
find an $n$-by-$n$ positive doubly stochastic matrix $C$ such that its eigenvalues are
$\lambda^*_1, \lambda^*_2, \ldots , \lambda^*_n$.}

Doubly stochastic matrices are crucial for many applications including  communication theory of  satellite-switched, time-division, multiple-access systems \cite{B88}, quantum mechanics \cite{L97}, graph theory \cite{B88,M88} (e.g., critical arcs for strongly connected digraphs \cite{HS72}), graph-based clustering \cite{VH16,WH16,YO12,ZS06}, and  the assignment problem \cite{M12}, etc.

The IEP for doubly stochastic matrices aims to find a  doubly stochastic matrix from the prescribed spectrum. The existence theory is an interesting question and some necessary or sufficient conditions were provided in some literature (e.g., \cite{HP04,LX14,M84,M03,PM65}). Some constructive  methods were proposed for solving
the IEP for doubly stochastic matrices \cite{MA13,PM65,RR03}. Recently, there have been some Riemannian optimization methods for solving the IEP for different structure matrices such as  Riemannian nonlinear conjugate gradient methods for solving the IEP for doubly stochastic matrices \cite{YB16}, the IEP for  stochastic matrices \cite{ZJB16},  and a Riemannian inexact Newton-CG method for solving the IEP for nonnegative matrices \cite{ZBJ18}.

In this paper, we propose both monotone and nonmonotone   Riemannian inexact Newton-CG methods for solving the PDStIEP. This is motivated by the recent two papers due to Zhao, Bai, and Jin \cite{ZBJ18} and Li and Fukushima \cite{LF99}. In \cite{ZBJ18}, a Riemannian inexact Newton-CG method was provided for solving the IEP for nonnegative matrices, where the global and quadratic convergence was established under some assumptions. In \cite{LF99},  Li and Fukushima presented a Gauss-Newton-based BFGS method for solving symmetric
nonlinear equations. By exploring the real Schur decomposition and the geometric properties of the set of  positive doubly stochastic matrices, the PDStIEP is written as a nonlinear matrix equation on a matrix product manifold. Then we give both monotone and nonmonotone   Riemannian inexact Newton-CG methods for solving the nonlinear matrix equation with a constraint of a matrix product manifold. The global and quadratic convergence of our methods is  derived under some assumptions. We also compute invariant subspaces of the constructed solution to the PDStIEP via its real Schur decomposition. Finally, we present  some numerical tests, including an application in the digraph, to illustrate the efficiency of the proposed methods.

Throughout this paper, we use the following notation. Let $\Rmn$ be the set of all $m\times n$ real matrices,
which is equipped with  the Frobenius inner product
\[
\langle A,B\rangle_F=\tr(A^TB)\quad \forall A,B\in\Rmn
\]
and its induced Frobenius norm $\|\cdot\|_F$, where ``\tr" denotes the trace of a square matrix. We use $I_n$ to denote the identity matrix of order $n$. Let  $\mathbf{0}_n$ and $\mathbf{0}_{n\times n}$  be the $n$-dimensional zero vector and the zero matrix of order $n$, respectively.  For any two matrices $A,B\in\Rmn$, $A\oslash B$ and $A\odot B$ stand for the Hadamard division and Hadamard product of $A$ and $B$, respectively. For any  matrix $A\in\Rmn$,  we let $(A)_{ij}$ or $a_{ij}$ stand for the $(i,j)$-entry of $A$, where $1\le i\le m$ and $1\le j\le n$. In addition, $[A,B]:=AB-BA$ mean Lie Bracket of  two square matrices $A$ and $B$.
We use $\exp(A)$ to denote the entry-wise exponential of a matrix $A$. Let $|\cdot|$ be the absolute value of a real or complex number. Let $\Rnn_{++}$ denote the set of all real $n$-by-$n$ element-wise positive matrices.

The rest of this paper is organized as follows. In Section \ref{sec:2} we reformulate the PDStIEP as a nonlinear matrix equation on a matrix product manifold and propose both monotone and nonmonotone   Riemannian inexact Newton-CG methods for solving the nonlinear matrix equation. In Section \ref{sec:3} we derive the global and quadratic convergence of our methods under some assumptions. In Section \ref{sec:4} we
further explore how to compute invariant subspaces of the solution to the PDStIEP once its real Schur decomposition is available. Finally, some numerical tests and concluding remarks are reported in  Sections \ref{sec:5} and \ref{sec:6}, respectively.

\section{Riemannian inexact Newton-CG methods}\label{sec:2}
In this section, by exploring the real Schur decomposition, we reformulate the PDStIEP as a nonlinear matrix equation on a matrix product manifold. Based on  the geometric properties of the set of positive doubly stochastic matrices, we present both monotone and nonmonotone  Riemannian inexact Newton-CG methods for solving the nonlinear matrix equation.
\subsection{Reformulation}
An $n$-by-$n$ real matrix $C$ is called a positive doubly stochastic matrix if it is an element-wise positive matrix with each row and column summing to $1$. The set of all $n$-by-$n$ positive stochastic matrices is defined by
\BE\label{dpn}
\mathds{DP}_n:=\{ A\in \Rnn \ | \  A \in \Rnn_{++},
\; A\mathbf{e} = \mathbf{e}, \; \mbox{and} \; A^T\be = \be \},
\EE
where  $\be$ is an $n$-vector of all ones.

We point out that the {\em realizable list} of $n$ complex numbers $\{\lambda^*_1, \lambda^*_2, \ldots, \lambda^*_n\}$ means that there exists at least an $n$-by-$n$ positive doubly stochastic matrix with  $\{\lambda^*_1, \lambda^*_2, \ldots, \lambda^*_n\}$  as its spectrum \cite{JM17}. Since the set $\{\lambda^*_1, \lambda^*_2, \ldots, \lambda^*_n\}$ is closed under complex conjugation, without loss of generality, we may assume that
\[
\la_{2j-1}^*=a_j+ b_j\ri,\quad \la_{2j}^*=a_j- b_j\ri, \quad j=1,\ldots,s; \quad \la_j^*\in\R, \quad j=2s+1,\ldots,n,
\]
where $a_j, b_j\in\R$ with $b_j\neq 0$ for $j=1,\ldots,s$ and $\ri=\sqrt{-1}$. Then we define a block diagonal matrix by
\[
\La:={\rm blkdiag}\left(\la_1^{[2]*}, \ldots,\la_s^{[2]*},\la_{2s+1}^*,\ldots,\la_n^*\right)
\]
with diagonal blocks $\la_1^{[2]*}, \ldots,\la_s^{[2]*},\la_{2s+1}^*,\ldots,\la_n^*$,
where
\[
\lambda_j^{[2]*} =
\left[
\begin{array}{cc}
a_j & 0\\
0 & a_j
\end{array}
\right],\quad j=1,\ldots,s.
\]

By using  \cite[Theorem 2.3.4]{RHCJ12}, we know that, for a real $n\times n$ matrix $A$, there exists an $n\times n$ orthogonal matrix $Q$ such that
\[
A=QTQ^T,
\]
where $T$ is a real  $n\times n$ upper quasitriangular matrix (i.e., the real Schur form) with $2\times 2$ and $1\times 1$ blocks on the diagonal. The $1\times 1$ blocks are the real eigenvalues of $A$ and the eigenvalues of these $2\times 2$ blocks are the complex conjugate eigenvalues  of $A$. As noted in  \cite{BD93} and \cite{Brandts02}, for $1\le j\le s$, the $j$-th diagonal $2\times 2$ block of $T$ can be standardized in the form of
\[
\left[
\begin{array}{cc}
a & b\\
c & a
\end{array}
\right],\quad bc< 0
\]
via a Givens rotation based orthogonal similarity transformation. Sparked by this, define the sets $\co(n)$,  $\cv$, and $\cw$  by
\BE\label{ovw}
\left\{
\begin{array}{lcl}
\co(n) & : =&  \big\{ Q \in \Rnn \ | \ Q^TQ = I_{n} \big\},\\[2mm]
\cv &: = & \big\{ V \in \Rnn  \ | \  v_{ij}= 0, \; (i,j) \in \mathcal{I}_1 \cup \mathcal{I}_2 \big\},\\[2mm]
\cw &: =&  \big\{ W \in \Rnn  \ | \  w_{ij}= 0  \; \mbox{for} \;  (i,j) \not\in \mathcal{I}_2
\; \mbox{and} \; w_{ij} > 0\; \mbox{for} \; (i,j) \in \mathcal{I}_2\big\},
\end{array}
\right.
\EE
where $\mathcal{I}_1$ and  $\mathcal{I}_2$ are two index subsets defined by
\[
\mathcal{I}_1: = \big\{ (i,j) \ | \ i\geq j,  \; i,j =1,\ldots,n \big\} \quad \mbox{and}  \quad
\mathcal{I}_2:=\big\{ (2k-1,2k) \ |\ k =1,\ldots,s \big\}.
\]
Also, define a linear operator $\ca:\cw \to \Rnn$ by
\[
\big(\ca(W)\big)_{ij}:=
\left\{
\begin{array}{ll}
\displaystyle \frac{-b_{i/2}^2}{w_{ji}}, & (j,i)\in \mathcal{I}_2, \\[4mm]
  0, &  \mbox{otherwise},
\end{array}
\right.
\]
for all $W\in \cw$. Then it is easy to check that the following matrix set
\[
\cm: =  \big\{ Q(\Lambda +\ca(W) + W + V) Q^T \ | \  Q\in \co(n),\; W\in \cw,\; V\in \cv \big\}
\]
consists of all real $n\times n$ matrices with the prescribed eigenvalues $\lambda^*_1, \lambda^*_2, \ldots , \lambda^*_n$. Hence,  the PDStIEP has a solution if and only if $\cm\cap  \mathds{DP}_n \neq \emptyset$.

In what follows, we assume that the PDStIEP has at least one solution. The PDStIEP is equivalent to solving the following nonlinear matrix equation
\BE\label{DStIEPEQUATION}
F(C,Q,W,V)= \mathbf{0}_{n\times n}
\EE
for $(C,Q,W,V)\in \mathds{DP}_n \times \co(n) \times \cw \times \cv$, where the mapping $F: \mathds{DP}_n \times \co(n) \times \cw  \times \cv \to\Rnn$ is defined by
\BE\label{DStIEP:G}
F(C,Q,W,V)= C - Q(\Lambda + \ca(W) +  W + V)Q^T,
\EE
for all $(C,Q,W,V)\in \mathds{DP}_n \times \co(n) \times \cw \times \cv$.
\subsection{Geometric properties of $\mathds{DP}_n \times \co(n) \times \cw \times \cv$}
We first discuss the set $\mathds{DP}_n$ defined by \eqref{dpn}. As noted in \cite{DH18}, $\mathds{DP}_n$ is a multinomial manifold. Obviously, the doubly stochastic multinomial manifold $\mathds{DP}_n$ is a submanifold of $\Rnn_{++}$. The tangent space of $\mathds{DP}_n$  at $A\in  \mathds{DP}_n$ is given by
\[
T_A \mathds{DP}_n = \{ \xi_A \in \Rnn \ | \  \xi_A \be = \mathbf{0}_n,\;
\xi_A^T \be = \mathbf{0}_n \}.
\]
Let $\mathds{DP}_n$ be endowed with the  Fisher information metric  \cite{LL05}:
\BE\label{Fishermetric}
\langle \xi_A, \eta_A \rangle: = {\rm tr}\big( (\xi_A \oslash A)\eta_A^T \big)
= \sum_{i=1}^n \sum_{j=1}^n \frac{(\xi_{A})_{ij}(\eta_{A})_{ij}}{a_{ij}},
\quad \forall A\in\mathds{DP}_n, \;\xi_A,\eta_A\in T_A \mathds{DP}_n.
\EE
Then $\mathds{DP}_n$ is a Riemannian submanifold of $\Rnn_{++}$,
whose dimension is $(n-1)^2$ \cite{DH18}.
Let $A\in  \mathds{DP}_n$. With respect to the Riemannian metric (\ref{Fishermetric}),
the orthogonal projection $\Pi_A: \Rnn \to T_A \mathds{DP}_n$
is given by \cite[Theorem 2]{DH18}:
\[
\Pi_A B=  B- \big( \alpha \be^T + \be \beta^T\big) \odot A, \quad \forall B\in \Rnn.
\]
where the vectors $\alpha$ and $\beta$ are determined by  the following linear system:
\[
\left(
   \begin{array}{cc}
     I_n & A \\
     A^T & I_n \\
   \end{array}
 \right)
 \left(
 \begin{array}{c}
   \alpha \\
   \beta
 \end{array}
  \right)=\left(
 \begin{array}{c}
   B\be \\
   B^T\be
 \end{array}
\right).
\]
Let $\mathcal{P} :\Rnn_{++} \to \mathds{DP}_n$ denote the mapping from the set of element-wise positive matrices to the set of
doubly stochastic matrices  via the Sinkhorn-Knopp algorithm \cite{Sinkborn64}.
A retraction $R$ on $\mathds{DP}_n$ is a mapping from the tangent bundle $T\mathds{DP}_n=\cup_{A\in\mathds{DP}_n}T_A \mathds{DP}_n$ onto $\mathds{DP}_n$, which  can be chosen as \cite[Lemma 5]{DH18}:
\BE
 R_A(\xi_A): =  \mathcal{P}\big(A \odot  \exp(\xi_A \oslash A)\big),
\quad \forall  \xi_A\in T_A \mathds{DP}_n.
\EE

Next, we study the set $\co(n)$ defined by \eqref{ovw}.  It is easy to see that  $\co(n)$ is the set of all $n$-by-$n$ orthogonal matrices, which is an orthogonal group. The tangent space of $\co(n)$  at $Q\in\co(n)$ is given by \cite[p. 42]{AMS08}
\[
T_Q \co(n) = \{ QK \ | \  K^T=-K, \; K\in\Rnn \}.
\]
The Riemannian metric on  $\co(n)$ is inherited from the standard inner product of $\Rnn$, i.e.,
\[
\langle \xi_Q, \eta_Q \rangle := {\rm \tr}(\xi_Q^T\eta_Q), \quad \forall Q\in\co(n),\;\xi_Q,\eta_Q\in T_Q\co(n).
\]
Then $\co(n)$ is an embedded Riemannian submanifold of $\Rnn$,
whose dimension is $n(n-1)/2$.  A retraction $R$ on $\co(n)$  can be chosen as
\[
R_Q(\xi_Q)=\qf(Q+\xi_Q),\quad \forall Q\in\co(n), \;\xi_Q\in T_Q \co(n).
\]
Here, $\qf(\cdot)$ is the $Q$ factor of the QR decomposition of a nonsingular square matrix, where the $R$ factor has positive diagonal entries.
For more choices of retractions on $\co(n)$, one may refer to \cite[p. 58]{AMS08}.

We now discuss  the set $\cv$  defined by \eqref{ovw}.  It is obvious that $\cv$ is a subspace of $\Rnn$ and the tangent space of $\cv$ at a point  $V\in\cv$ is given by
\[
T_V\cv = \cv.
\]
Let $\cv$ be equipped with the Frobenius inner product $\langle\cdot,\cdot\rangle_F$ on $\Rnn$. Then $\cv$ is a Riemannian submanifold of $\Rnn$, whose dimension is $n(n-1)/2 - s$. For any $V\in\cv$, the orthogonal projection $\Pi_V: \Rnn \to T_W \cv$
is given by
\[
\Pi_V Z=  S\odot Z, \quad \forall Z\in \Rnn,
\]
where the matrix $S\in\Rnn$ is defined by
\[
S_{ij} :=
\left\{
\begin{array}{ll}
0, &  (i,j) \in \mathcal{I}_1 \cup \mathcal{I}_2 ,\\[2mm]
1, & \mbox{otherwise}.
\end{array}
\right.
\]
The retraction $R$ on $\cv$ is given by
\[
R_V (\xi_V) = V + \xi_V, \quad \forall   V\in \cv,\; \xi_V\in T_V\cv.
\]

In the following, we focus on the set $\cw$ defined by \eqref{ovw}.
We observe that $\cw$ is a submanifold of $\Rnn$, whose dimension is $s$.
The tangent space of $\cw$ at $W\in\cw$ is given by
\[
T_W\cw  = \big\{ W \in \Rnn  \ | \  W_{ij}= 0  \; \mbox{for} \;  (i,j) \not\in \mathcal{I}_2\big\}.
\]
Let $\cw$ be endowed with the following Fisher information metric
\BE\label{Fishermetric22}
\langle \xi_W, \eta_W \rangle: = \sum_{(i,j)\in \mathcal{I}_2} \frac{(\xi_{W})_{ij}(\eta_{W})_{ij}}{W_{ij}},
\quad \forall W\in\cw,\;\xi_W,\eta_W\in T_W \cw.
\EE
For any  $W\in\cw$, with respect to the Riemannian metric (\ref{Fishermetric22}),
the orthogonal projection $\Pi_W: \Rnn \to T_W \cw$
is given by
\[
\Pi_W Z=  M\odot Z, \quad \forall Z\in \Rnn,
\]
where the matrix $M\in \Rnn$ is given by
\[
M_{ij} :=
\left\{
\begin{array}{ll}
1, &  (i,j) \in \mathcal{I}_2 ,\\[2mm]
0, & \mbox{otherwise}.
\end{array}
\right.
\]
A retraction $R$ on $\cw$ is given by
\[
\big(R_W(\xi_W)\big)_{ij} :=
\left\{
\begin{array}{ll}
W_{ij}\exp\big((\xi_W)_{ij}/W_{ij} \big), &  (i,j) \in \mathcal{I}_2 ,\\[2mm]
0, & \mbox{otherwise},
\end{array}
\right.
\]
for all $W\in \cw$ and $\xi_W\in T_W\cw$.

Based on the above analysis, we give the basic geometric properties of the matrix product manifold $\mathds{DP}_n \times \co(n) \times \cw \times \cv$.
Let
\[
\cz:=\mathds{DP}_n \times \co(n) \times \cw \times \cv.
\]
Then the tangent space of $\cz$ at a point $Z:=(C,Q,W,V)\in\cz$ is given by
\[
T_Z\cz  = T_C\mathds{DP}_n\times T_Q\co(n) \times T_W\cw\times T_V\cv
\]
and the product manifold  $\cz$ is endowed with the Riemannian metric
\BE\label{rm:z}
\langle \xi_Z, \eta_Z \rangle: = \langle \xi_C, \eta_C \rangle + \langle \xi_Q, \eta_Q \rangle + \langle \xi_W, \eta_W \rangle +\langle \xi_V, \eta_V \rangle_F,
\EE
for all $Z=(C,Q,W,V)\in\cz,\;\xi_Z=(\xi_C,\xi_Q,\xi_W,\xi_V),\eta_Z=(\eta_C,\eta_Q,\eta_W,\eta_V)\in T_Z\cz$.
We note that
\[
{\rm dim}(\cz)
= \displaystyle  (n-1)^2 + \frac{n(n-1)}{2} + s + \frac{n(n-1)}{2} - s
=\displaystyle  (2n-1)(n-1)\quad\mbox{and}\quad
{\rm dim}(\Rnn) = n^2.
\]
Therefore, the nonlinear equation $F(Z)= \mathbf{0}_{n\times n}$ defined by \eqref{DStIEPEQUATION} is underdetermined on the product manifold  $\cz$ for $n\geq 3$.

A  retraction $R$ on $\cz$ is given by
\BE\label{r:z}
R_\cz(\xi_Z)=(R_C(\xi_C),R_Q(\xi_Q),R_W(\xi_W),R_V(\xi_V)),
\EE
for all $Z=(C,Q,W,V)\in\cz$ and $\xi_Z=(\xi_C,\xi_Q,\xi_W,\xi_V)\in T_Z\cz$.

In the rest of this subsection, we derive the differential of the nonlinear operator $F$ defined by \eqref{DStIEP:G}. The differential
$\mathrm{D}F(Z):T_{Z}\cz \to T_{F(Z)}\Rnn$
of $F$ at a point $Z=(C,Q,W,V)\in\cz$ is determined by
\begin{eqnarray*}
\mathrm{D}F(Z)[\Delta Z] &=& \Delta C + [ Q(\Lambda+ \ca(W) + W + V)Q^T , \Delta QQ^T]\\
&&- Q\big( (B_W\odot \Delta W)^T + \Delta W + \Delta V\big) Q^T,
\end{eqnarray*}
for all $\Delta Z:=(\Delta C,\Delta Q, \Delta W, \Delta V) \in T_Z\cz$, where  the matrix $B_W\in\Rnn$ is defined by
\[
(B_W)_{ij}:=
\left\{
\begin{array}{ll}
\displaystyle  \frac{b_{j/2}^2}{w^2_{ij}}, & (i,j)\in \mathcal{I}_2, \\[4mm]
  0, &  \mbox{otherwise}.
\end{array}
\right.
\]
With respect to the Riemannian metric \eqref{rm:z} on $\cz$ and the Frobenius inner product on  $\Rnn$, via simple calculation, the adjoint $(\mathrm{D}F(Z))^*:T_{F(Z)}\Rnn \to T_{Z}\cz$ of  $\mathrm{D}F(Z)$ is determined  by
\BE\label{dh:adj}
(\mathrm{D}F(Z))^* [ \Delta Y ] =\Big((\mathrm{D}F(Z))^*_1 [ \Delta Y ],(\mathrm{D}F(Z))^*_2 [ \Delta Y ], (\mathrm{D}F(Z))^*_3 [ \Delta Y ], (\mathrm{D}F(Z))^*_4 [ \Delta Y ]\Big)
\EE
for all $\Delta Y \in T_{F(Z)}\Rnn\simeq\Rnn$, where
\BE\label{dhadj}
\left\{
\begin{array}{rcl}
(\mathrm{D}F(Z))^*_1 [ \Delta Y ] &=& \Pi_C(C\odot \Delta Y),\\[2mm]
(\mathrm{D}F(Z))^*_2 [ \Delta Y ] &=& \displaystyle
\frac{1}{2}\Big([ Q(\Lambda +  \ca(W) + W + V)Q^T, (\Delta Y)^T ] \\[2mm]
& & + [ Q(\Lambda + \ca(W) + W + V)^TQ^T, \Delta Y ]\Big)Q,\\[2mm]
(\mathrm{D}F(Z))^*_3 [ \Delta Y ] &=& -W\odot  \big((Q^T\Delta Y Q) + B_W\odot (Q^T\Delta Y^T Q)\big),\\[2mm]
(\mathrm{D}F(Z))^*_4 [ \Delta Y ] &=& -S\odot\big(Q^T\Delta Y Q\big).
\end{array}
\right.
\EE

\subsection{Riemannian inexact Newton-CG methods}
In this subsection, we present both monotone and nonmonotone Riemannian inexact Newton-CG methods for solving  \eqref{DStIEPEQUATION}.
For any $Z\in\cz$, let $\mathrm{id}_{T_{F(Z)} \Rnn}$ be the identity operator on $T_{F(Z)} \Rnn$.
To solve the underdetermined nonlinear equation \eqref{DStIEPEQUATION}, we first adopt the Riemannian inexact Newton-CG method proposed in \cite{ZBJ18}. The algorithm can be stated as follows.

\begin{algorithm} \label{nm}
 {\rm (\bf{Monotone Riemannian inexact Newton-CG method})}
\begin{description}
\item [{Step 0.}] Choose a starting  point $Z_0\in \cz$, $\epsilon>0$, $\overline{\sigma}_{\max}, \overline{\eta}_{\max}, t\in (0,1)$, $0<\theta_{\min}<\theta_{\max}<1$. Let $k:=0$.

\item [{Step 1.}] If $\|F(Z_k)\|_F<\epsilon$, stop.
\item [{Step 2.}] Apply the conjugate gradient (CG) method  {\rm \cite{GV13}} to solving
    \BE\label{eq:mle}
    \big(\mathrm{D}F(Z_k)\circ (\mathrm{D}F(Z_k))^* +\overline{\sigma}_k \mathrm{id}_{T_{F(Z_k)} \Rnn} \big)[\Delta Y_k] = - F(Z_k),
    \EE
    for $\Delta Y_k \in T_{F(Z_k)} \Rnn$ such that
   \BE\label{eq:mtol1}
    \|\big( \mathrm{D}F(Z_k) \circ (\mathrm{D}F(Z_k))^* +\overline{\sigma}_k \mathrm{id}_{T_{F(Z_k)} \Rnn}\big)[ \Delta Y_k ]
    + F(Z_k)\big\|_{F} \le  \overline{\eta}_k\|  F(Z_k) \|_{F} ,
    \EE
    and
    \BE\label{eq:mtol2}
    \|\mathrm{D}F(Z_k) \circ (\mathrm{D}F(Z_k))^* [ \Delta Y_k ] + F(Z_k)\|_{F} <  \|  F(Z_k) \|_{F} ,
    \EE
    where $\overline{\sigma}_k:=\min\{\overline{\sigma}_{\max},\| F(Z_k)\|_{F}\}$, $\overline{\eta}_k:=\min\{\overline{\eta}_{\max},\| F(Z_k)\|_{F}\}$. Set
    \[
    \widehat{\Delta Z}_k= (\mathrm{D}F(Z_k))^*[ \Delta Y_k ],\quad
    \widehat{\eta}_k:= \frac{\|\mathrm{D}F(Z_k) [ \widehat{\Delta Z}_k ] + F(Z_k)\|_{F}}{\| F(Z_k)\|_{F}}.
    \]

\item [{Step 3.}] Evaluate $F\big(R_{Z_k}(\widehat{\Delta Z}_k)\big)$.
                  Set $\eta_k = \widehat{\eta}_k$ and $\Delta Z_k = \widehat{\Delta Z}_k$.

                 {\bf Repeat} until $\|F\big(R_{Z_k}(\Delta Z_k)\big)\|_F\le (1-t(1-\eta_k)) \|F(Z_k)\|_F$.

                  \qquad\quad Choose $\theta \in [\theta_{\min},\theta_{\max}]$.

                  \qquad\quad Replace  $\Delta Z_k$ by $\theta \Delta Z_k $ and $\eta_k$ by $1-\theta(1-\eta_k)$.

                  {\bf end (Repeat)}

                  Set
                  \[
                  Z_{k+1} := R_{Z_k} (\Delta Z_k ).
                  \]
\item [{Step 4.}] Replace $k$ by $k+1$ and go to  {\bf Step 1}.
\end{description}
\end{algorithm}

On Algorithm \ref{nm} for the PDStIEP  \eqref{DStIEPEQUATION}, we have the following remark. Define the merit function
\BE\label{mf:z}
f(Z):=\frac{1}{2}\|F(Z)\|_F^2.
\EE
The Riemannian gradient of $f$ at a point $Z_k\in\cz$  is given by \cite[p.185]{AMS08}:
\BE\label{GRAVEC}
\grad f(Z_k) = (DF(Z_k))^*[F(Z_k)].
\EE
We note that  the linear equation \eqref{eq:mle} is solved such that the condition \eqref{eq:mtol2} is satisfied. Then using \eqref{eq:mtol2} we have
\begin{eqnarray*}
&&\langle\widehat{\Delta Z}_k,\grad f(Z_k)\rangle=\langle(\mathrm{D}F(Z_k))^*[ \Delta Y_k ],(DF(Z_k))^*[F(Z_k)]\rangle \\
&=&\langle \mathrm{D}F(Z_k)\circ(DF(Z_k))^*[\Delta Y_k],F(Z_k)\rangle_F \\
&=& \langle \mathrm{D}F(Z_k)\circ(DF(Z_k))^*[\Delta Y_k]+F(Z_k),F(Z_k)\rangle_F-\|F(Z_k)\|_F^2<0.
\end{eqnarray*}
Hence, $\widehat{\Delta Z}_k= (\mathrm{D}F(Z_k))^*[ \Delta Y_k ]$ is a descent direction of $f$.
However, it is too strict to solve \eqref{eq:mle} satisfying both \eqref{eq:mtol1} and \eqref{eq:mtol2}. As a classical inexact Newton method, it is  natural  to solve \eqref{eq:mle} satisfying only \eqref{eq:mtol1}. In this case, the search direction $\widehat{\Delta Z}_k= (\mathrm{D}F(Z_k))^*[ \Delta Y_k ]$ may be just an approximate Newton direction of $f$ at $Z_k$. This means that $\widehat{\Delta Z}_k$ is not necessarily a descent direction of $f$ at $Z_k$ especially  when $\mathrm{D} F(Z_k)$ is not surjective and thus the monotone line search in  Step 3 of Algorithm \ref{nm} may not be satisfied.  Sparked by the line  search strategy in \cite{LF99}, we propose the following nonmonotone Riemannian inexact Newton-CG method for solving  the PDStIEP  \eqref{DStIEPEQUATION}. Here, we provide a new nonmonotone line search as follows. Let $\{\gamma_k>0\}$ be a sequence such that
\[
\sum_{k=0}^{\infty}\gamma_k = \gamma <\infty.
\]
We determine the stepsize $\psi>0$ such that
\BE\label{nondes}
\|F\big(R_{Z_k}(\psi\Delta Z_k)\big)\|_F^2 - \|F(Z_k)\|_F^2  \leq
      -\delta \psi^2 | \langle  \grad f(Z_k), \Delta Z_k \rangle|  + \gamma_k\|F(Z_k)\|_F^2,
\EE
where  $\delta>0$ is a constant. We point out that, when $\psi\to 0^{+}$, the left-hand side of (\ref{nondes}) tends to zero,
while the right hand side tends to the positive constant $\gamma_k\|F(Z_k)\|_F^2$.
Thus the line search step  determined by (\ref{nondes})  is well-defined.

Based on the above analysis, we describe a nonmonotone  Riemannian inexact Newton-CG algorithm as follows.
\begin{algorithm} \label{nm1}
 {\rm (\bf Nonmonotone  Riemannian inexact Newton-CG method)}
\begin{description}
\item [{Step 0.}] Choose a starting point $Z_0 \in \cz$, $\epsilon>0$, $\tau$, $\rho$, $\overline{\sigma}_{\max} \in (0,1)$,
                  $\delta \in \big(0,\frac{1}{2}\big)$, and two positive sequences
                  $\{\eta_k\}$ and $\{\gamma_k\}$ such that
                  \BE\label{etakepk}
                  \lim_{k\to \infty} \eta_k = 0 \quad \mbox{and} \quad
                  \sum_{k=0}^{\infty}\gamma_k = \gamma <\infty.
                  \EE
Let $k:=0$.
\item [{Step 1.}] If $\|F(Z_k)\|_F<\epsilon$, stop.
\item [{Step 2.}] Apply the CG method to find an approximate solution $\Delta Y_k\in T_{F(Z_k)} \Rnn$ to
    \BE\label{eq:le}
    ( \mathrm{D}F(Z_k) \circ (\mathrm{D}F(Z_k))^* +\overline{\sigma}_k \mathrm{id}_{T_{F(Z_k)} \Rnn})  [ \Delta Y_k ]
      = - F(Z_k)
    \EE
    such that
    \BE\label{eq:tol1}
    \begin{array}{l}
    \|( \mathrm{D}F(Z_k) \circ (\mathrm{D}F(Z_k))^* +\overline{\sigma}_k \mathrm{id}_{T_{F(Z_k)} \Rnn})[ \Delta Y_k ]
    + F(Z_k)\|_F\leq  \overline{\eta}_k\|  F(Z_k) \|_F ,
    \end{array}
    \EE
    where
    \BE\label{def:skek}
    \overline{\sigma}_k:=\min\{\overline{\sigma}_{\max},\| F(Z_k)\|_F\} \quad \mbox{and} \quad \overline{\eta}_k:=\min\{\eta_k,\| F(Z_k)\|_F\}.
    \EE
Set
\BE\label{eq:direstep}
    \Delta Z_k:= (\mathrm{D}F(Z_k))^*[ \Delta Y_k ].
 \EE

\item [{Step 3.}]
    If
    \BE\label{FDESCENTV}
    \big\|F\big(R_{Z_k}(\Delta Z_k)\big)\big\|_F \leq \tau \|F(Z_k)\|_F,
    \EE
    then set $\alpha_k=1$;  Otherwise, determine the stepsize $\alpha_k:=\max\{ \rho^l,l=0,1,2,\ldots\}$ such that
    \BE\label{DES}
      \|F\big(R_{Z_k}(\alpha_k\Delta Z_k)\big)\|_F^2 - \|F(Z_k)\|_F^2  \leq
      -\delta \alpha_k^2 | \langle  \grad f(Z_k), \Delta Z_k \rangle|  + \gamma_k\|F(Z_k)\|_F^2.
    \EE
    Set
    \BE\label{def:NEWPP}
    Z_{k+1} := R_{Z_k}(\alpha_k\Delta Z_k).
    \EE

\item [{Step 4.}] Replace $k$ by $k+1$ and go to {\bf Step 1}.
\end{description}
\end{algorithm}
\begin{remark}
From the convergence analysis in  {\rm  Section \ref{sec:3}} below, we observe that the global and quadratic convergence of
{\rm Algorithm \ref{nm1}} can be established under much milder assumptions than  {\rm  Algorithm \ref{nm}}. We also see that the  infinite sequence generated by   {\rm Algorithm \ref{nm1}} converges to a stationary point of $f$ without any additional assumption.
\end{remark}

\section{Convergence analysis}\label{sec:3}
In this section, we establish global and quadratic convergence of Algorithms \ref{nm} and \ref{nm1}. We first note that the global and quadratic convergence of  Algorithm \ref{nm} can be established as in \cite{ZBJ18} under the following assumption:
\begin{assumption} \label{ass:nons}
Suppose {\rm Algorithm \ref{nm}} does not break down, $\sum_{k=0}^{\infty}(1-\eta_k)$ is divergent
and $\mathrm{D} F(Z_*):T_{Z_*}\cz\to T_{F(Z_*)}\Rnn$ is surjective,
where $Z_*\in\cz$ is an accumulation point of the sequence $\{Z_k\}$ generated by {\rm Algorithm  \ref{nm}}.
\end{assumption}

In the rest of this section, we focus on the convergence analysis of  Algorithm \ref{nm1}.
The pullback $\widehat{f}: T\cz \to  \R$ of $f$ defined by \eqref{mf:z} with respect to the retraction $R$ \eqref{r:z} on $\cz$ is defined by \cite[p.55]{AMS08}:
\[
\widehat{f}(\xi)= f(R(\xi)),\quad \forall \xi\in T\cz:=\cup_{Z\in\cz}T_Z\cz.
\]
For any $Z\in\cz$, $\widehat{f}_Z: T_Z\cz \to\R$ denotes the restriction of $\widehat{f}$ to $T_Z \cz$ \cite[(4.3)]{AMS08}, i.e.:
\BE \label{def:pbr}
\widehat{f}_Z (\xi_Z)= f(R_Z(\xi_Z)),\quad \forall \xi_Z\in T_Z\cz.
\EE
By the local rigidity of $R$, we have \cite[(4.4)]{AMS08}:
\BE\label{GRADWIDEf}
\grad \widehat{f}_Z(0_Z) = \grad f(Z), \quad \forall Z\in \cz.
\EE

Let $\Omega$ denote the level set of $\|F(Z)\|_F$ defined by
\BE\label{LEVELSET}
\Omega: = \big\{Z\in \cz \ | \ \|F(Z)\|_F \leq e^{\frac{\gamma}{2}}\|F(Z_0)\|_F \big\}.
\EE
Since (\ref{FDESCENTV}) or (\ref{DES}) holds, we have
\BE
\|F(Z_{k+1})\|_F \leq (1+\epsilon_k)^{1/2} \|F(Z_k)\|_F \leq (1+\epsilon_k) \|F(Z_k)\|_F,\quad \forall k\ge 1.
\EE

We note that the  doubly stochastic multinomial manifold $\mathds{DP}_n$ and the orthogonal group $\co(n)$ are  compact and the retractions on $\cw$ and $\cv$ are exponential retractions. Then  there exist two scalars $\nu >0$ and $\mu_{\nu} >0$ such that \cite[p. 149]{AMS08}
\BE\label{retr:bd-2}
\nu\| \Delta Z \| \geq  {\rm dist}\big(Z,R_{Z} (\Delta Z) \big),
\EE
for all $Z\in \cz$ and $\Delta Z\in T_Z\cz$ with $\| \Delta Z\|\leq \mu_{\nu}$,
where ``{\rm dist}" means the Riemannian distance on $\cz$.

We first give the main results on  the global and  quadratic convergence of Algorithm {\rm \ref{nm1}}.
\begin{theorem}\label{thm:gc1}
Suppose {\rm Algorithm \ref{nm1}} generates an infinite  sequence $\{Z_{k}\}$. Then every accumulation point $Z_*$ of $\{Z_k\}$
is a stationary point of $f$.
\end{theorem}

\begin{theorem}\label{thm:gc2}
Let $Z_*$ be an accumulation point of
an infinite  sequence $\{Z_{k}\}$ generated by {\rm Algorithm {\rm \ref{nm1}}}.
If $\mathrm{D} F(Z_*):T_{Z_*}\cz \to T_{F(Z_*)}\Rnn$ is surjective, then the sequence  $\{Z_k\}$ converges to $Z_*$ and $F(Z_*) = {\bf 0}_{n\times n}$.
\end{theorem}

On the quadratic convergence of  Algorithm {\rm \ref{nm1}}, we have the following result.
\begin{theorem}\label{thm:qc}
Let $Z_*$ be an accumulation point of an infinite sequence $\{Z_{k}\}$ generated by {\rm Algorithm {\rm \ref{nm1}}}. If $\mathrm{D} F(Z_*):T_{Z_*}\cz \to T_{F(Z_*)}\Rnn$ is surjective, then the sequence  $\{Z_k\}$ converges to $Z_*$ quadratically.
\end{theorem}

Next, we establish the global and  quadratic convergence of Algorithm {\rm \ref{nm1}}. First, we have the following result on the convergence of $\{\|F(Z_k)\|_F\}$. The proof is similar to \cite[Lemma 3.1]{LF99} and thus we omit it here.
\begin{lemma}\label{pro:FDESCENT}
Let $\{Z_k\}$ be a sequence generated by  {\rm Algorithm \ref{nm1}}. Then $\{Z_k\}$ is contained in $\Omega$. Moreover, the sequence $\{\|F(Z_k)\|_F\}$ converges, i.e.,
$\lim_{k\to \infty}\|F(Z_k)\|_F$ exists.
\end{lemma}

The following lemma shows that the series $\sum_{k=0}^{\infty}  \alpha_k^2 |\langle  \grad f(Z_k), \Delta Z_k \rangle|$ is convergent under some mild condition.
\begin{lemma}\label{lemma:akzkconverges}
If the inequality {\rm (\ref{FDESCENTV})} is satisfied for only a finite number of outer iterations, then we have
\[
\sum_{k=0}^{\infty}  \alpha_k^2 |\langle  \grad f(Z_k), \Delta Z_k \rangle|  < \infty.
\]
\end{lemma}

\begin{proof}
Suppose the inequality {\rm (\ref{FDESCENTV})} is satisfied for  only a finite number of outer iterations. Then $\alpha_k$ is determined by (\ref{DES}) for all $k$ sufficiently large.
From (\ref{DES}) and (\ref{def:NEWPP}) we have for all $k$ sufficiently large,
\[
\delta \alpha_k^2 | \langle  \grad f(Z_k), \Delta Z_k \rangle |
\leq \|F(Z_k)\|_F^2 - \|F(Z_{k+1})\|_F^2 + \gamma_k\|F(Z_k)\|_F^2.
\]
Since $\sum_{k=0}^{\infty} \gamma_k  < \infty$ and $\{\|F(Z_k)\|_F\}$ is bounded,
the convergence of $\sum_{k=0}^{\infty}  \alpha_k^2 |\langle  \grad f(Z_k), \Delta Z_k \rangle|$  can be obtained by summing the above inequalities.
\end{proof}

By following the similar proof of \cite[Lemmas 2--3]{ZBJ18}, we have the following lemma on the iterate  $\Delta Z_{k}$ generated by Algorithm {\rm \ref{nm1}}.
\begin{lemma}\label{lem:dzk}
Let $Z_{k}$ be the current iterate generated by {\rm  Algorithm\ref{nm1}}. Then we have
\[
\|\Delta Z_k\| \leq (1+\eta_k)\| ({\rm D}F(Z_k))^{\dag}\|\|F(Z_k)\|_F
\]
and
\[
 \|F(Z_k) + {\rm D}F(Z_k)[\Delta Z_k]\|_F
\leq \left(\frac{\sigma_k}{\sigma_k + \lambda_{\min}\big({\rm D}F(Z_k)\circ ({\rm D}F(Z_k))^*\big) } + \eta_k \right)\|F(Z_k)\|_F,
\]
where $\lambda_{\min}(\cdot)$ denotes the smallest eigenvalues of a self-adjoint linear operator and $(\mathrm{D}F(X))^{\dag}$ is the pseudoinverse of $\mathrm{D}F(X)$ {\rm \cite[p.163]{LU69}}.
\end{lemma}

The following lemma shows that the sequences $\{\Delta Z_{k}\}$ and  $\{ \langle  \grad f(Z_k), \Delta Z_k\rangle\}$  generated by Algorithm {\rm \ref{nm1}} have some accumulation points under some condition.
\begin{lemma}\label{lemma23}
Suppose  {\rm Algorithm \ref{nm1}} generates
an infinite sequence $\{Z_{k}\}$. Let $Z_*$ be an accumulation point  of $\{Z_k\}$ and $\{Z_k\}_{k\in \mathcal{K}}$ be a subsequence of $\{Z_k\}$ converging to $Z_*$. If $\lim_{k \to \infty,k\in \mathcal{K}} \|F(Z_k)\|_F  >0$, then we have
\[
\lim\limits_{k \to \infty,k\in \mathcal{K}} \Delta Z_k =
-(\mathrm{D}F(Z_*))^* \circ\big (\mathrm{D}F(Z_*) \circ (\mathrm{D}F(Z_*))^*
+\bar{\sigma} \mathrm{id}_{T_{F(Z_*)} \Rnn} \big)^{-1} [F(Z_*)],
\]
and
\begin{eqnarray*} \label{DELTAZKFKI}
&&\lim\limits_{k \to \infty,k\in \mathcal{K}}  \langle  \grad f(Z_k), \Delta Z_k\rangle \\
&=& -\big\langle F(Z_*), \mathrm{D}F(Z_*)\circ (\mathrm{D}F(Z_*))^* \circ \big(\mathrm{D}F(Z_*) \circ (\mathrm{D}F(Z_*))^* +\bar{\sigma} \mathrm{id}_{T_{F(Z_*)} \Rnn} \big)^{-1}[F(Z_*)] \big\rangle \\
&\ge&-\frac{1}{\bar{\sigma}}\|F(Z_*)\|_F^2,
\end{eqnarray*}
where $\bar{\sigma}:=\lim_{k \to \infty,k\in \mathcal{K}} \sigma_k$.
\end{lemma}

\begin{proof}
By the hypothesis, $\lim_{k \to \infty,k\in \mathcal{K}}\|F(Z_k)\|_F>0$. Then there exists a constant $c>0$ such that
\BE\label{BOUNDBELOWC}
\|F(Z_k)\|_F \geq c, \quad \forall k \in \mathcal{K}.
\EE
Since $\lim_{k \to \infty,k\in \mathcal{K}} Z_k = Z_*$ and $F$ is continuously differentiable, we have
\BE\label{FDFCC}
\lim\limits_{k \to \infty,k\in \mathcal{K}} {\rm D}F(Z_k) = {\rm D}F(Z_*), \quad \mbox{and} \quad
\lim\limits_{k \to \infty,k\in \mathcal{K}} ({\rm D}F(Z_k))^* = ({\rm D}F(Z_*))^*.
\EE
Using (\ref{def:skek}) and (\ref{BOUNDBELOWC}) we have
\BE\label{sigmaklbd}
\bar{\sigma}=\lim_{k \to \infty,k\in \mathcal{K}} \sigma_k \geq \min\{ \sigma_{\max}, c\}>0.
\EE
Let
\BE\label{def:VK}
W(Z_k) := \big( \mathrm{D}F(Z_k) \circ (\mathrm{D}F(Z_k))^* +\sigma_k \mathrm{id}_{T_{F(Z_k)} \Rnn}\big)
[ \Delta Z_k ]  + F(Z_k).
\EE
From  (\ref{etakepk}), (\ref{eq:tol1}), (\ref{def:skek}), (\ref{def:VK}), and Lemma \ref{pro:FDESCENT} we have
\BE\label{VKLIMIT}
\lim_{k\to \infty} W(Z_k) =  {\bf 0}_{n\times n}.
\EE
Using (\ref{eq:le}), (\ref{eq:tol1}), and (\ref{def:VK}) we have
\BE\label{DELTAZK}
\Delta Z_k =  \big(\mathrm{D}F(Z_k) \circ (\mathrm{D}F(Z_k))^* +\sigma_k \mathrm{id}_{T_{F(Z_k)} \Rnn}\big)^{-1}
[W(X_k) - F(Z_k)].
\EE
It follows from (\ref{GRAVEC}), (\ref{eq:direstep}), (\ref{FDFCC}), (\ref{sigmaklbd}),
(\ref{VKLIMIT}), and (\ref{DELTAZK}) that
\begin{eqnarray*}
& &\lim\limits_{k \to \infty,k\in \mathcal{K}} \Delta Z_k \\
&=& \lim\limits_{k \to \infty,k\in \mathcal{K}} \big(\mathrm{D}F(Z_k))^*\circ (\mathrm{D}F(Z_k) \circ (\mathrm{D}F(Z_k))^* +\sigma_k \mathrm{id}_{T_{F(Z_k)} \Rnn}\big)^{-1}[W(Z_k) - F(Z_k)]\\
&=&(\mathrm{D}F(Z_*))^* \circ \big(\mathrm{D}F(Z_*) \circ (\mathrm{D}F(Z_*))^*
+\bar{\sigma} \mathrm{id}_{T_{F(Z_*)} \Rnn}\big)^{-1} [F(Z_*)].
\end{eqnarray*}
and
\begin{eqnarray*}
& &\lim\limits_{k \to \infty,k\in \mathcal{K}} \langle \grad f(Z_k), \Delta Z_k  \rangle
= \lim\limits_{k \to \infty,k\in \mathcal{K}} \langle (\mathrm{D}F(Z_k))^*[F(Z_k)], \Delta Z_k  \rangle\\
&=& \lim\limits_{k \to \infty,k\in \mathcal{K}} \langle  F(Z_k), \mathrm{D}F(Z_k)[\Delta Z_k]  \rangle
=\langle  \lim\limits_{k \to \infty,k\in \mathcal{K}} F(Z_k),
\lim\limits_{k \to \infty,k\in \mathcal{K}} \mathrm{D}F(Z_k)[\Delta Z_k]  \rangle\\
&=&-\big\langle F(Z_*), \mathrm{D}F(Z_*)\circ (\mathrm{D}F(Z_*))^* \circ \big(\mathrm{D}F(Z_*) \circ (\mathrm{D}F(Z_*))^* +\bar{\sigma} \mathrm{id}_{T_{F(Z_*)} \Rnn} \big)^{-1}[F(Z_*)] \big\rangle_F\\
&\geq&-\frac{1}{\bar{\sigma}}\|F(Z_*)\|_F^2.
\end{eqnarray*}
The proof is complete.

\end{proof}

We now give the proof of Theorem \ref{thm:gc1}.

\vspace{3mm}
{\noindent \bf \em Proof of Theorem \ref{thm:gc1}}
If the equality (\ref{FDESCENTV}) holds for  an infinitely number of outer iterations, then we have
$\lim_{k\to \infty}\|F(Z_k)\|_F=0$.
In this case, every accumulation point of $\{Z_k\}$ is a stationary point of $f$.
Thus we only need to consider the case where (\ref{FDESCENTV}) is satisfied for only a finite number of outer iterations.
In this case, the stepsize $\alpha_k$ is determined by (\ref{DES}) for all $k$ sufficiently large.

Let $Z_{*}$ be an accumulation point of the sequence $\{Z_{k}\}$. Then there exists a subsequence $\{Z_k\}_{k\in \mathcal{K}}$ such that
$\lim_{k \to \infty,k\in \mathcal{K}} Z_{k}=Z_{*}$.
By Lemma \ref{lemma:akzkconverges} we have
\BE\label{af}
\lim \limits_{k \to \infty}  \alpha_k^2 \langle  \grad f(Z_k), \Delta Z_k\rangle  =0.
\EE
If $\liminf_{k\to\infty}\alpha_{k}>0$, then it follows from (\ref{af}) that
\BE\label{CASE1}
\lim\limits_{k \to \infty} \langle  \grad f(Z_k), \Delta Z_k\rangle =0.
\EE
We claim that $\lim_{k \to \infty,k\in \mathcal{K}}\|F(Z_k)\|_F=0$.
By contrary, if $\lim_{k \to \infty,k\in \mathcal{K}}\|F(Z_k)\|_F>0$,
then it follows from Lemma \ref{lemma23} that
$\lim_{k \to \infty,k\in \mathcal{K}} \langle  \grad f(Z_k), \Delta Z_k\rangle <0$,
which is a contradiction to (\ref{CASE1}).
Therefore,  $\lim_{k \to \infty,k\in \mathcal{K}}\|F(Z_k)\|_F=0$.

On the other hand, if $\liminf_{k\to\infty}\alpha_{k}=0$, then there exists a subsequence $\{\alpha_{k}\}_{k\in\mathcal{K}_1}$ of
the sequence $\{\alpha_k\}_{k\in \mathcal{K}}$ such that $\lim_{k \to \infty,k\in \mathcal{K}_1}\alpha_{k} = 0$.
If $\lim_{k\to \infty}\|F(Z_k)\|_F=0$, then the conclusion holds.
Thus we only need to consider the case that $\lim_{k\to \infty}\|F(Z_k)\|_F>0$.
By using Lemma \ref{lemma23} and  (\ref{DES}) we have for $k\in \mathcal{K}_1$ sufficiently large,
\[
\begin{array}{rcl}
\displaystyle \Big\|F\Big(R_{Z_{k}}\Big(\frac{\alpha_k}{\rho}\Delta Z_{k}\Big)\Big)\Big\|_F^2 - \|F(Z_{k})\|_F^2
&\geq& \displaystyle
-\delta\frac{\alpha_{k}^2}{\rho^2}|\langle  \grad f(Z_k), \Delta Z_k\rangle|  + \gamma_{k}\|F(Z_{k})\|_F^2\\
&\geq& \displaystyle \delta\frac{\alpha_{k}^2}{\rho^2}\langle  \grad f(Z_k), \Delta Z_k\rangle
\geq -\frac{2\delta}{\bar{\sigma}}\frac{\alpha_{k}^2}{\rho^2}\|F(Z_*)\|_F^2.
\end{array}
\]
This, together with (\ref{mf:z}) and (\ref{def:pbr}), yields
\[
\begin{array}{rcl}
\displaystyle \widehat{f}_{Z_k}\Big(\frac{\alpha_k}{\rho}\Delta Z_k\Big) - \widehat{f}_{Z_k}(0_{Z_k})
&=&\displaystyle f\Big(R_{Z_k}\Big(\frac{\alpha_k}{\rho}\Delta Z_k\Big)\Big) - f(Z_k) \\[2mm]
&=& \displaystyle \frac{1}{2} \Big\|F\Big(R_{Z_k}\big(\frac{\alpha_k}{\rho}\Delta Z_k\Big)\Big)\Big\|_F^2 - \frac{1}{2}\|F(Z_k)\|_F^2
\geq  \displaystyle -\frac{\delta}{\bar{\sigma}}\frac{\alpha_{k}^2}{\rho^2}\|F(Z_*)\|_F^2.
\end{array}
\]
Thus,
\[
\begin{array}{rcl}
\displaystyle \frac{\widehat{f}_{Z_k}\Big(\frac{\alpha_k}{\rho}\Delta Z_k\Big) - \widehat{f}_{Z_k}(0_{Z_k})}{\frac{\alpha_k}{\rho}}
&\geq& \displaystyle  -\frac{\delta}{\bar{\sigma}}\frac{\alpha_{k}}{\rho}\|F(Z_*)\|_F^2.
\end{array}
\]
By using the mean-value theorem, there exists a positive constant $\theta_k \in (0,1)$ such that
\BE\label{DES-CC}
\begin{array}{rcl}
\displaystyle \Big\langle  \grad \widehat{f}_{Z_k}\Big(\theta_k \frac{\alpha_k}{\rho}\Delta Z_k\Big),
\Delta Z_k \Big\rangle
&\geq& \displaystyle -\frac{\delta}{\bar{\sigma}}\frac{\alpha_{k}}{\rho}\|F(Z_{*})\|_F^2.
\end{array}
\EE
By using Lemma \ref{lemma23}, we know that the sequence $\{\Delta Z_k\}_{k\in \mathcal{K}}$ converges.
Let $\Delta Z_*:=\lim_{k \to \infty,k\in \mathcal{K}} \Delta Z_k$.
Using (\ref{GRADWIDEf}) and (\ref{DES-CC}) we find
\BE\label{GRADDELTAX1}
\begin{array}{rcl}
\langle \grad f(Z_*), \Delta Z_* \rangle =
\langle \grad \widehat{f}_{Z_*}(0_{Z_*}), \Delta Z_* \rangle \geq 0.
\end{array}
\EE
Using Lemma \ref{lemma23} we have
\[
\begin{array}{rcl}
& &\langle \grad f(Z_*), \Delta Z_* \rangle
=  -\langle  ({\rm D}F(Z_*))^*[F(Z_*)],   \Delta Z_*  \rangle \\[2mm]
&=& -\big\langle F(Z_*), \mathrm{D}F(Z_*)\circ (\mathrm{D}F(Z_*))^* \circ \big(\mathrm{D}F(Z_*) \circ (\mathrm{D}F(Z_*))^* +\bar{\sigma} \mathrm{id}_{T_{F(Z_*)} \Rnn} \big)^{-1}[F(Z_*)] \big\rangle \leq 0.
\end{array}
\]
This, together with (\ref{GRADDELTAX1}), implies that
\[
\big\langle F(Z_*), \mathrm{D}F(Z_*)\circ (\mathrm{D}F(Z_*))^* \circ \big(\mathrm{D}F(Z_*) \circ
(\mathrm{D}F(Z_*))^* +\bar{\sigma} \mathrm{id}_{T_{F(Z_*)} \Rnn}\big)^{-1} [F(Z_*)] \big\rangle = 0.
\]
Since $F(Z_*)\neq 0$, it follows from (\ref{sigmaklbd}) and the above equality that
\BE\label{GFGFT1}
F(Z_*) \; \bot \; {\rm im}({\rm D}F(Z_*)).
\EE
In addition, we have
\BE\label{GFGFT2}
{\rm ker}(({\rm D}F(Z_*))^*) \; \bot \; {\rm im}({\rm D}F(Z_*)).
\EE
Based on (\ref{GFGFT1}) and (\ref{GFGFT2}), we have $ F(Z_*)\in {\rm ker}(({\rm D}F(Z_*))^*)$, i.e.,
$({\rm D}F(Z_*))^*[F(Z_*)]=0_{Z_*}$. Then it follows from (\ref{GRAVEC}) that $\grad f(Z_*) = 0_{Z_*}$.
Thus the proof is complete. \hfill $\fbox{}$ \vspace*{5mm}

Next, we give the proof of Theorem \ref{thm:gc2}.

\vspace{3mm}
{\noindent \bf \em Proof of Theorem \ref{thm:gc2}}
By hypothesis, $Z_*$ is an accumulation point of an infinite  sequence $\{Z_{k}\}$ generated by Algorithm {\rm \ref{nm1}}. Then by  Theorem \ref{thm:gc1}, we know  that $Z_*$ is a stationary point of $f$, i.e.,
\[
\grad f(Z_*) = ({\rm D}F(Z_*))^*[F(Z_*)] = 0_{Z_*}.
\]
By assumption $\mathrm{D} F(Z_*):T_{Z_*}\cz \to \Rnn$ is surjective. Then the above equality implies that
\BE\label{FXZERO}
F(Z_*) = {\bf 0}_{n\times n}.
\EE
By using Lemma \ref{pro:FDESCENT} and (\ref{FXZERO}) we have
\BE\label{FXZERONM}
\lim_{k\to\infty} F(Z_k) = {\bf 0}_{n\times n}.
\EE

Since $F$ is continuously differentiable and $\mathrm{D} F(Z_*)$ is surjective,
there exists a positive constant $\delta_0>0$ such that
for all $X\in B(Z_*,\delta_0)$,
\BE\label{DFGENINVERSEEST1}
\lambda_{\min}\big({\rm D}F(X)\circ ({\rm D}F(X))^*\big)
\geq \frac{1}{2} \lambda_{\min}\big({\rm D}F(Z_*)\circ ({\rm D}F(Z_*))^*\big) >0
\EE
and
\BE\label{DFGENINVERSEEST2}
\|(\mathrm{D} F(X))^{\dag}\| \leq  2\|(\mathrm{D} F(Z_*))^{\dag}\|.
\EE
Based on Lemma \ref{lem:dzk}, (\ref{eq:tol1}), (\ref{def:skek}), and (\ref{DFGENINVERSEEST2}) we obtain
\BE\label{THEOREMC21}
\|\Delta Z_k\| \leq (1+\eta_k)\| ({\rm D}F(Z_k))^{\dag}\|\|F(Z_k)\|_F
\leq  2\|(\mathrm{D} F(Z_*))^{\dag}\|\cdot \|F(Z_k)\|_F
\EE
for all $Z_k\in B(Z_*,\delta_0)$.
Since $F$ is continuously differentiable, there exist two positive constants $\delta_1\leq \delta_0$ and $\mu_1< \mu_{\nu}$
such that
\BE\label{TYALOREXP}
\|F(R_X(\Delta X)) - F(X) - {\rm D}F(X)[\Delta X] \|_F \leq \frac{\tau}{4\|(\mathrm{D} F(Z_*))^{\dag}\|} \|\Delta X\|
\EE
for $X\in B(Z_*,\delta_1)$ and $\|\Delta X\| \leq \mu_1$.
By using (\ref{FXZERONM}) and (\ref{THEOREMC21}), there exists a positive constant $\delta_2\leq \delta_1$ such that
\BE\label{DELTAXKBD}
 \|\Delta Z_k \| \leq \mu_1 < \mu_{\nu}, \quad  \forall Z_k \in  B(Z_*,\delta_2).
\EE
Thus it follows from (\ref{THEOREMC21}) and (\ref{TYALOREXP}) that
\BE\label{TYALOREXP22}
\|F(R_{Z_k}(\Delta Z_k)) - F(Z_k) - {\rm D}F(Z_k)[\Delta Z_k] \|_F \leq \frac{\tau}{2} \|F(Z_k)\|_F,
\quad \forall Z_k \in  B(Z_*,\delta_2).
\EE

Let
\BE\label{tauk}
\tau_k := \frac{\|F(Z_k) + {\rm D}F(Z_k)[\Delta Z_k]\|_F}{\|F(Z_k)\|_F}.
\EE
Using Lemma \ref{lem:dzk}, (\ref{def:skek}),  (\ref{DFGENINVERSEEST1}),  and (\ref{tauk}) we have
\begin{eqnarray}\label{T221}
\tau_k
&\leq&\displaystyle  \frac{\sigma_k}{\sigma_k + \lambda_{\min}\big({\rm D}F(Z_k)\circ ({\rm D}F(Z_k))^*\big) } + \eta_k \nonumber \\
&\leq&\displaystyle   \left(1+\frac{2}{\lambda_{\min}\big({\rm D}F(X*)\circ ({\rm D}F(Z_*))^*\big) }\right)\|F(Z_k)\|_F,
\quad \forall Z_k \in  B(Z_*,\delta_0).
\end{eqnarray}
By (\ref{FXZERONM}) and (\ref{T221}), there exists a positive constant $\delta_3\leq \delta_0$ such that
\BE\label{T2212}
\tau_k < \frac{\tau}{2}, \quad  \forall Z_k \in  B(Z_*,\delta_3).
\EE
Let $\hat{\delta} = \min\{\delta_2,\delta_3\}$.
Based on (\ref{TYALOREXP22}),  (\ref{tauk}), and (\ref{T2212}), we have
\begin{eqnarray*}
\|F(R_{Z_k}(\Delta Z_k))\|_F
&=& \|F(R_{Z_k}(\Delta Z_k)) - F(Z_k) - {\rm D}F(Z_k)[\Delta Z_k]
+ F(Z_k) + {\rm D}F(Z_k)[\Delta Z_k]\|_F \\
&\leq& \|F(R_{Z_k}(\Delta Z_k)) - F(Z_k) - {\rm D}F(Z_k)[\Delta Z_k]\|_F
+ \|F(Z_k) + {\rm D}F(Z_k)[\Delta Z_k]\|_F \\
&\leq& \tau \|F(Z_k)\|_F, \quad    \forall Z_k \in  B(Z_*, \hat{\delta}).
\end{eqnarray*}
This, together with {\rm (\ref{FDESCENTV})} and (\ref{def:NEWPP}), yields
\BE\label{DES2221}
\left\{
\begin{array}{rcl}
&&Z_{k+1} = R_{Z_k}(\Delta Z_k), \\[2mm]
&& \|F(Z_{k+1})\|_F \leq \tau \|F(Z_k)\|_F=\left[ 1 - (1-\tau)\right] \|F(Z_k)\|_F,
\quad \forall Z_k \in  B(Z_*, \hat{\delta}).
\end{array}
\right.
\EE

We now show that $\{Z_k\}$ converges to $Z_*$. By contradiction, assume that $\{Z_k\}$ does not converge to $Z_*$.
Then there exist infinitely many $k$ such that $Z_k \not\in  B_{\hat{\delta}}(Z_*)$.
Since $Z_*$ is an accumulation point of $\{Z_k\}$, there exist two index sets
 $\{m_j\}$ and $\{n_j\}$ such that $\lim_{j\to \infty}Z_{m_j} = Z_*$, and for each $j$,
\[
\left\{
\begin{array}{rcl}
Z_{m_j} &\in& B_{\hat{\delta}/2}(Z_*), \quad Z_{m_j+i} \in B_{\hat{\delta}}(Z_*), \quad i=0,\ldots,n_j-1, \\[2mm]
Z_{m_j+n_j} &\not\in& B_{\hat{\delta}}(Z_*), \quad m_j + n_j  < m_{j+1}.
\end{array}
\right.
\]
Then, using (\ref{retr:bd-2}), (\ref{FXZERONM}), (\ref{THEOREMC21}), (\ref{DELTAXKBD}),  and (\ref{DES2221}) we have
\[
\begin{array}{rcl}
\displaystyle \frac{\hat{\delta}}{2} &\leq& \mbox{dist}(Z_{m_j+n_j}, Z_{m_j})
\leq \sum\limits^{m_j+n_j-1}_{k=m_j} \mbox{dist}(Z_{k+1}, Z_{k})  \\[2mm]
&=&  \displaystyle \sum^{m_j+n_j-1}_{k=m_j} \mbox{dist}\big(R_{Z_{k}}(\Delta Z_k), Z_{k} \big)
\leq \sum^{m_j+n_j-1}_{k=m_j} \nu \| \Delta Z_k \| \\[2mm]
&\leq& \displaystyle \sum^{m_j+n_j-1}_{k=m_j} 2\nu \|(\mathrm{D} F(Z_*))^{\dag}\|\cdot \|F(Z_k)\|_F
= \displaystyle \sum^{m_j+n_j-1}_{k=m_j} \frac{2\nu \|(\mathrm{D} F(Z_*))^{\dag}\|}{1-\tau} (1-\tau) \|F(Z_k)\|_F \\[2mm]
&\leq& \displaystyle \sum^{m_j+n_j-1}_{k=m_j} \frac{2\nu \|(\mathrm{D} F(Z_*))^{\dag}\|}{1-\tau}
(\|F(Z_k)\|_F - \|F(Z_{k+1})\|_F)\\[2mm]
&=&  \displaystyle \frac{2\nu \|(\mathrm{D} F(Z_*))^{\dag}\|}{1-\tau} \big( \|F(Z_{m_j})\|_F - \|F(Z_{m_j+n_j})\|_F\big) \\[2mm]
&\to& 0,\quad \mbox{as }\; j\to\infty.
\end{array}
\]
This is a contradiction. Thus the sequence $\{Z_k\}$ converges to $Z_*$. This completes the proof. \hfill $\fbox{}$ \vspace*{5mm}

Finally, we give the proof of Theorem \ref{thm:qc}.

\vspace{3mm}
{\noindent \bf \em Proof of Theorem \ref{thm:qc}}
This follows directly from the proof of  \cite[Theorem 3]{ZBJ18}. \hfill $\fbox{}$ \vspace*{5mm}

\section{Invariant subspace computations} \label{sec:4}
In this section, we further compute  invariant subspaces of an $n$-by-$n$ positive doubly stochastic matrix $C_*$ when its real Schur form is available. By  Algorithm \ref{nm} or Algorithm \ref{nm1} we can obtain a solution to the PDStIEP \eqref{DStIEPEQUATION}. That is, from the prescribed eigenvalues $\lambda^*_1, \lambda^*_2, \ldots , \lambda^*_n$, we can find an
$n$-by-$n$ positive doubly stochastic matrix $C$ with a  real Schur form
\BE\label{dc:rs}
Q_*^TC_*Q_*=\Lambda +  \ca(W_*) + W_* + V_*\equiv T.
\EE
Denote
\BE\label{mat:bd}
T=\begin{blockarray}{ccccc}
n_1 & n_2 &\cdots & n_q &\\
\begin{block}{[cccc]c}
T_{11} & T_{12} & \cdots & T_{1q} & n_1\\
0 &  T_{12} & \cdots & T_{2q} & n_2\\
\vdots &  \vdots & \ddots & \vdots & \vdots \\
0 &  0 & \cdots & T_{qq} & n_q\\
\end{block}
\end{blockarray}\equiv (T_{ij}),
\EE
where $\la(T_{ii})\cap\la(T_{jj})=\emptyset$ whenever $i\neq j$. Here, $\la(\cdot)$ denotes the spectrum of a square matrix. By using \cite[Theorem 7.1.6]{GV13}, we can find a nonsingular matrix $Y\in\Rnn$ such that
\BE\label{bd:t}
Y^{-1}TY=\diag(T_{11},\ldots,T_{qq}),
\EE
where $\diag(T_{11},\ldots,T_{qq})$ is a block diagonal matrix with diagonal blocks $T_{11},\ldots,T_{qq}$.

Let $I_n=[E_1, E_2, \ldots, E_q]$ with $E_i\in\R^{n\times n_i}$ for $i=1,\ldots,q$.
As noted in \cite[section 7.6.3]{GV13}, one may determine the $Y=\prod_{1\le i<j\le q} Y_{ij}$, where
\[
Y_{ij}=I_n+E_iZ_{ij}E_j^T,\quad Z_{ij}\in\R^{n_i\times n_j}.
\]
Let $\overline{T}=T$. Then we update by $\overline{T}=Y_{ij}^{-1}\overline{T}Y_{ij}\equiv (\overline{T}_{ij})$, where
\begin{eqnarray*}
\overline{T}_{ij} &=&\overline{T}_{ii}Z_{ij}-Z_{ij}\overline{T}_{jj}+\overline{T}_{ij}= {\bf 0}_{n_i\times n_j},\\
\overline{T}_{ik} &=&\overline{T}_{ik}-Z_{ij}\overline{T}_{jk},\quad k=j+1:q.
\end{eqnarray*}
Here, the block $Z_{ij}$ is determined by the Sylvester equation
\[
\overline{T}_{ii}Z_{ij}-Z_{ij}\overline{T}_{jj}=-\overline{T}_{ij},
\]
which can be solved by the   Bartels-Stewart algorithm (\cite{BS72} and   \cite[Algorithm 7.6.2]{GV13}).
On the invariant subspace computation of $C$, we have the following algorithm, which comes from  \cite[Algorithm 7.6.3]{GV13}.
\begin{algorithm}  \label{isc}
{\rm \bf Invariant Subspace Computations}
\begin{description}
\item [{\rm Step 0.}] Given a real Schur form \eqref{dc:rs} of a  positive doubly stochastic matrix $C_*\in\Rnn$, where $Q_*\in\Rnn$ is an orthogonal matrix and $T\in\Rnn$ is an upper quasi-triangular matrix with the form of \eqref{mat:bd}. Let $\Theta:=Q_*$.
\item [{\rm Step 1.}] \mbox{for $j=2:q$}\\
\mbox{~~~~~~for $i=1:j-1$}\\
\mbox{~~~~~~~~~~~Solve $T_{ii}Z_{ij}-Z_{ij}T_{jj}=-T_{ij}$ for $Z_{ij}$} \\
\mbox{~~~~~~~~~~~for $k=j+1:q$}\\
\mbox{~~~~~~~~~~~~~~~$T_{ik} =T_{ik}-Z_{ij}T_{jk}$}\\
\mbox{~~~~~~~~~~~end}\\
\mbox{~~~~~~~~~~~for $k=1:q$}\\
\mbox{~~~~~~~~~~~~~~~$\Theta_{kj}=\Theta_{ki}Z_{ij}+\Theta_{kj}$}\\
\mbox{~~~~~~~~~~~end} \\
\mbox{~~~~~~end} \\
\mbox{~~end}
\end{description}
\end{algorithm}

From Algorithm \ref{isc},  we observe that $\Theta=QY$, where the nonsingular matrix $Y$  satisfies \eqref{bd:t}. Let $\Theta=[\Theta_1, \Theta_2, \ldots, \Theta_q]$ with $\Theta_i\in\R^{n\times n_i}$ for $i=1,\ldots,q$. Then we also have
\[
C\Theta_i=\Theta_i T_{ii}, \quad i=1,\ldots, q.
\]
This shows that, for any $1\le i\le q$, $\mathcal{R}(\Theta_i)$ forms the invariant subspace of $C$ corresponding to the eigenvalues determined by $T_{ii}$, where $\mathcal{R}(\Theta_i)$ is the subspaces spanned by the column vectors of $\Theta_i$.

\section{Numerical experiments}\label{sec:5}
In this section, we report the numerical tests of Algorithms \ref{nm} and \ref{nm1} for solving the PDStIEP \eqref{DStIEPEQUATION}.  Our numerical tests were carried out by using {\tt MATLAB 2020a} running on a workstation with  an Intel Xeon CPU E5-2687W of 3.10 GHz and 32 GB of RAM.

We consider the following two numerical examples.
\begin{example}\label{ex:1}
We consider the  {\rm PDStIEP} with arbitrary eigenvalues. Let $\widetilde{C}$ be an $n\times n$ positive matrix with random entries uniformly distributed on the interval $(0,1)$.  Let $\widehat{C}=\cp(\widetilde{C})$ be a positive doubly stochastic matrix, which is obtained by the Sinkhorn-Knopp algorithm  {\rm \cite{Sinkborn64}}. We choose the eigenvalues of $\widehat{C}$ as the prescribed spectrum.
\end{example}
\begin{example}\label{ex:2}
We consider the  {\rm PDStIEP} with multiple zero eigenvalue. Let $\widetilde{C}=C_1C_2$, where $C_1\in\R^{n\times p}$ and $C_1\in\R^{p\times n}$ are two positive matrices with random entries uniformly distributed on the interval $(0,1)$.  Let $\widehat{C}=\cp(\widetilde{C})$ be a positive doubly stochastic matrix, which is obtained by the Sinkhorn-Knopp algorithm  {\rm \cite{Sinkborn64}}. We choose the eigenvalues of $\widehat{C}$ as the prescribed spectrum.
\end{example}

In our numerical tests,  for  Algorithms \ref{nm}  and \ref{nm1},  the  starting points $Z_0=(C_0,Q_0,W_0,V_0)\in\cz$ are generated randomly as follows: For Example \ref{ex:1},
\BE\label{sp-1}
\left\{
\begin{array}{c}
\widetilde{C}_0= {\tt rand}\,(n,n),  \;  C_0 = \cp(\widetilde{C}_0) \in\mathds{DP}_n, \; \mbox{$W_0\in\cw$ with $(W_0)_{ij}=|b_j|$ for $(i,j)\in\ci_2$},\\[2mm]
\big[ Q_0, \widetilde{V}_0\big] = \mbox{\tt schur}\,(C_0,{\rm 'real'}) , \;  V_0 = S\odot\widetilde{V}_0\in\cv,
\end{array}
\right.
\EE
while for Example \ref{ex:2},
\BE\label{sp-2}
\left\{
\begin{array}{c}
\widetilde{C}_0= {\tt rand}\,(n,p)*{\tt rand}\,(p,n),  \;  C_0 = \cp(\widetilde{C}_0) \in\mathds{DP}_n, \; \mbox{$W_0\in\cw$ with $(W_0)_{ij}=|b_j|$ for $(i,j)\in\ci_2$},\\[2mm]
\big[ Q_0,\widetilde{V}_0 \big] = \mbox{\tt schur}\,(C_0,{\rm 'real'}) , \;  V_0 = S\odot \widetilde{V}_0\in\cv.
\end{array}
\right.
\EE
The stopping criteria  are set to be
\[
\|F(Z_k)\|_F \le \epsilon\equiv 5.0\times 10^{-8},
\]
and the largest number of iterations in the CG method is set to be $n^2$.
In addition, we set  $\overline{\sigma}_{\max}=10^{-6}$, $\overline{\eta}_{\max}=0.1$, $\theta_{\min}=0.1$, $\theta_{\max}=0.9$, and $t=10^{-4}$ for  Algorithm \ref{nm} and we set $\tau=0.9$, $\rho=0.5$, $\overline{\sigma}_{\max}=10^{-6}$, $\delta=10^{-4}$, $\eta_k=1/(k+2)$,  and
$\gamma_k=1/(k+2)^2$ for  Algorithm \ref{nm1}.

For comparison purposes,  we use the symbols  `{\tt CT.}', {\tt IT.}', `{\tt NF.}', `{\tt NCG.}',   `{\tt Res.}', and  `{\tt grad.}' to denote the total computing time in seconds, the number of outer iterations, the number of function evaluations, the total number of inner CG iterations, the residual $\|F(Z_k)\|_F$, and the residual $\|\grad f(Z_k)\|$ at the final iterates of the corresponding algorithms accordingly.

The numerical results for Examples \ref{ex:1}--\ref{ex:2} are given in Tables \ref{table1}--\ref{table2}. We observe from Tables \ref{table1}--\ref{table2} that both Algorithm \ref{nm}  and Algorithm \ref{nm1} are very efficient for solving the PDStIEP with different problem sizes. As expected, the quadratic convergence is also observed.
\begin{table}[ht]\renewcommand{\arraystretch}{1.0} \addtolength{\tabcolsep}{2pt}
  \caption{Numerical results for Example \ref{ex:1}.}\label{table1}
  \begin{center} {\scriptsize
   \begin{tabular}[c]{|c|r|r|r|r|r|l|l|}
     \hline
Alg. & $n$  & {\tt CT.} & {\tt IT.} & {\tt NF.} &  {\tt NCG.} &  {\tt Res.}   &  {\tt grad.}  \\  \hline
Alg. \ref{nm} & 100  &      0.5048 s  & 6  & 7  & 169 & $2.79\times 10^{-9}$   & $1.63\times 10^{-9}$  \\
& 200  &      2.0663 s  & 6  & 7      & 230 & $9.82\times 10^{-10}$   & $7.88\times 10^{-10}$  \\
& 500  &       22.278 s  & 6  & 7  &  333  & $3.92\times 10^{-10}$   & $1.94\times 10^{-10}$  \\
& 800  &       01 m 33 s  & 6  & 7  &  369  & $2.65\times 10^{-9}$   & $6.81\times 10^{-10}$  \\
& 1000 &   04 m 16 s  & 7  & 8  & 572  & $2.87\times 10^{-12}$  &   $2.38\times 10^{-12}$  \\
& 1500 &   14 m 35 s  & 6  & 7  & 397  & $2.25\times 10^{-8}$  &   $9.06\times 10^{-9}$  \\
& 2000 &   57 m 17 s  & 7  & 8  & 520  & $3.67\times 10^{-9}$  &   $3.18\times 10^{-10}$  \\
\cline{2-8}\hline
Alg. \ref{nm1} & 100  &      0.4143 s  & 7  & 8  & 167 & $2.36\times 10^{-9}$   & $1.23\times 10^{-9}$  \\
& 200  &       2.9338 s  & 7  & 8  &  307  & $3.36\times 10^{-12}$  & $3.77\times 10^{-12}$  \\
& 500  &       21.659 s  & 7  &  8  &  311  & $7.64\times 10^{-10}$  & $2.47\times 10^{-10}$  \\
& 800  &      01 m 25 s  & 7  & 8  &  331  & $1.96\times 10^{-8}$   & $7.19\times 10^{-9}$  \\
& 1000 &     02 m 49 s  & 7  & 8  &  375  & $3.19\times 10^{-9}$  &  $9.69\times 10^{-10}$  \\
& 1500 &   18 m 02 s  & 9  & 10  & 510  & $3.35\times 10^{-9}$  &   $4.26\times 10^{-10}$  \\
& 2000 &   54 m 16 s  & 8  & 9  & 499  & $2.42\times 10^{-8}$  &   $5.54\times 10^{-9}$  \\
 \cline{2-8}\hline
  \end{tabular} }
  \end{center}
\end{table}

\begin{table}[ht]\renewcommand{\arraystretch}{1.0} \addtolength{\tabcolsep}{2pt}
  \caption{Numerical results for Example \ref{ex:2}.}\label{table2}
  \begin{center} {\scriptsize
   \begin{tabular}[c]{|c|c|r|r|r|r|r|l|l|}
     \hline
Alg. & $n$ & $p$ & {\tt CT.} & {\tt IT.} & {\tt NF.} &  {\tt NCG.} &  {\tt Res.}   &  {\tt grad.}  \\  \hline
Alg. \ref{nm} & 100 & 25 &  0.1525 s  & 4  & 5  & 43 & $1.08\times 10^{-11}$   & $6.64\times 10^{-12}$  \\
& 200  & 50 &      0.4772 s  & 4  & 5    & 36 & $4.09\times 10^{-10}$   & $2.78\times 10^{-10}$  \\
& 500  & 125 &     4.5894 s  & 4  & 5  &  59  & $1.05\times 10^{-12}$   & $9.40\times 10^{-13}$  \\
& 800  & 200 &     8.3675 s  & 3  & 4  &  27  & $4.34\times 10^{-9}$   & $3.97\times 10^{-10}$  \\
& 1000 & 250 &     16.743 s  & 3  & 4  & 27  & $9.68\times 10^{-9}$  &   $2.81\times 10^{-9}$  \\
& 1500 & 375 &   01 m 15 s  & 3  & 4  & 28  & $3.72\times 10^{-9}$  &   $1.65\times 10^{-10}$  \\
& 2000 & 500 &   05 m 02 s  & 3  & 4  & 39  & $1.79\times 10^{-9}$  &   $4.18\times 10^{-11}$  \\
\cline{2-9}\hline
Alg. \ref{nm1} & 100 & 25 &    0.1635 s  & 5  & 6  & 59 & $2.29\times 10^{-13}$   & $2.27\times 10^{-13}$  \\
& 200 & 50 &      0.4435 s  & 4  & 5  &  36  & $4.09\times 10^{-10}$  & $2.78\times 10^{-10}$  \\
& 500 & 125  &    4.5754 s  & 4  &  5  &  59  & $1.05\times 10^{-12}$  & $9.40\times 10^{-13}$  \\
& 800  & 200 &   8.3907 s  & 3  & 4  &  27  & $4.24\times 10^{-9}$   & $3.97\times 10^{-10}$  \\
& 1000 & 250 &    16.460 s  & 3  & 4  &  27  & $9.68\times 10^{-9}$  &  $2.81\times 10^{-9}$  \\
& 1500 & 375 &   01 m 16 s  & 3  & 4  & 28  & $3.72\times 10^{-9}$  &   $1.65\times 10^{-10}$  \\
& 2000 & 500 &   05 m 00 s  & 3  & 4  & 39  & $1.79\times 10^{-9}$  &   $4.18\times 10^{-11}$  \\
 \cline{2-9}\hline
  \end{tabular} }
  \end{center}
\end{table}

To further illustrate the effectiveness of  Algorithms \ref{nm}  and \ref{nm1}, we consider an application of the PDStIEP  in digraph \cite{M88,V00}.
\begin{example} \label{ex:3}
Let $G=(\widehat{V},\widehat{E})$ be a digraph, where $\widehat{V}=\{P_1,\ldots,P_n\}$ contains $n$ vertices and $\widehat{E}$ contains the arcs of $G$  {\rm \cite{M88,V00}}. Let $\widetilde{C}\in\Rnn$ be a nonnegative model of $G$, where each nonzero entry $(\widetilde{C})_{ij}$ denotes the directed arc $\overrightarrow{P_iP_j}$ directed from $P_j$ to $P_j$. As noted in  {\rm \cite{HS72,V00}}, a  digraph $G$ is strongly connected if and only if its associated matrix $\widetilde{C}$ is irreducible or there is an irreducible doubly stochastic matrix $\widehat{C}$ with positive main diagonal entries so that if $i\neq j$ then   $(\widehat{C})_{ij}>0$ if and only if there is an arc from $P_j$ to $P_j$.  In this example,  we assume that
\[
\widetilde{C}=
\left[
\begin{array}{cccccc}
1/40 & 7/8 & 1/40 & 1/40  & 1/40 & 1/40 \\
1/40& 1/40& 19/80 & 19/80 & 19/80 & 19/80 \\
1/6 & 1/6 & 1/6 & 1/6 & 1/6 & 1/6 \\
1/40& 1/40& 1/40 & 9/20 & 1/40 & 9/20 \\
1/40& 1/40& 1/40 & 9/20 & 1/40 & 9/20 \\
1/6 & 1/6 & 1/6 & 1/6 & 1/6 & 1/6
\end{array}
\right],
\]
which is a Google matrix. By using the Sinkhorn-Knopp algorithm  {\rm \cite{Sinkborn64}}, we obtain the following positive doubly stochastic matrix
\[
\widehat{C}=
\left[
\begin{array}{cccccc}
    0.0849  &  0.7646  &  0.0578   & 0.0175  &  0.0578   & 0.0175 \\
    0.0553  &  0.0142  &  0.3573   & 0.1080   & 0.3573   & 0.1080 \\
    0.3301   &  0.0849  &  0.2246   & 0.0679  &  0.2246  &  0.0679 \\
    0.0998  &  0.0257   & 0.0679   & 0.3694  &  0.0679  &  0.3694 \\
    0.0998  &  0.0257   & 0.0679  &  0.3694  &  0.0679  &  0.3694 \\
    0.3301   & 0.0849   & 0.2246  &  0.0679  &  0.2246  &  0.0679
\end{array}
\right].
\]
The digraphs corresponding to $\widetilde{C}$ and $\widehat{C}$ are displayed in  {\rm Figure \ref{fig51}}.
Then we use the eigenvalues $\{ 1.0000,  -0.0856 \pm 0.3336\ri, 0.0000, 0.0000, 0.0000 \}$ of $\widehat{C}$ as the prescribed spectrum.
\end{example}
\begin{figure}[ht]
\includegraphics[width=0.5\textwidth]{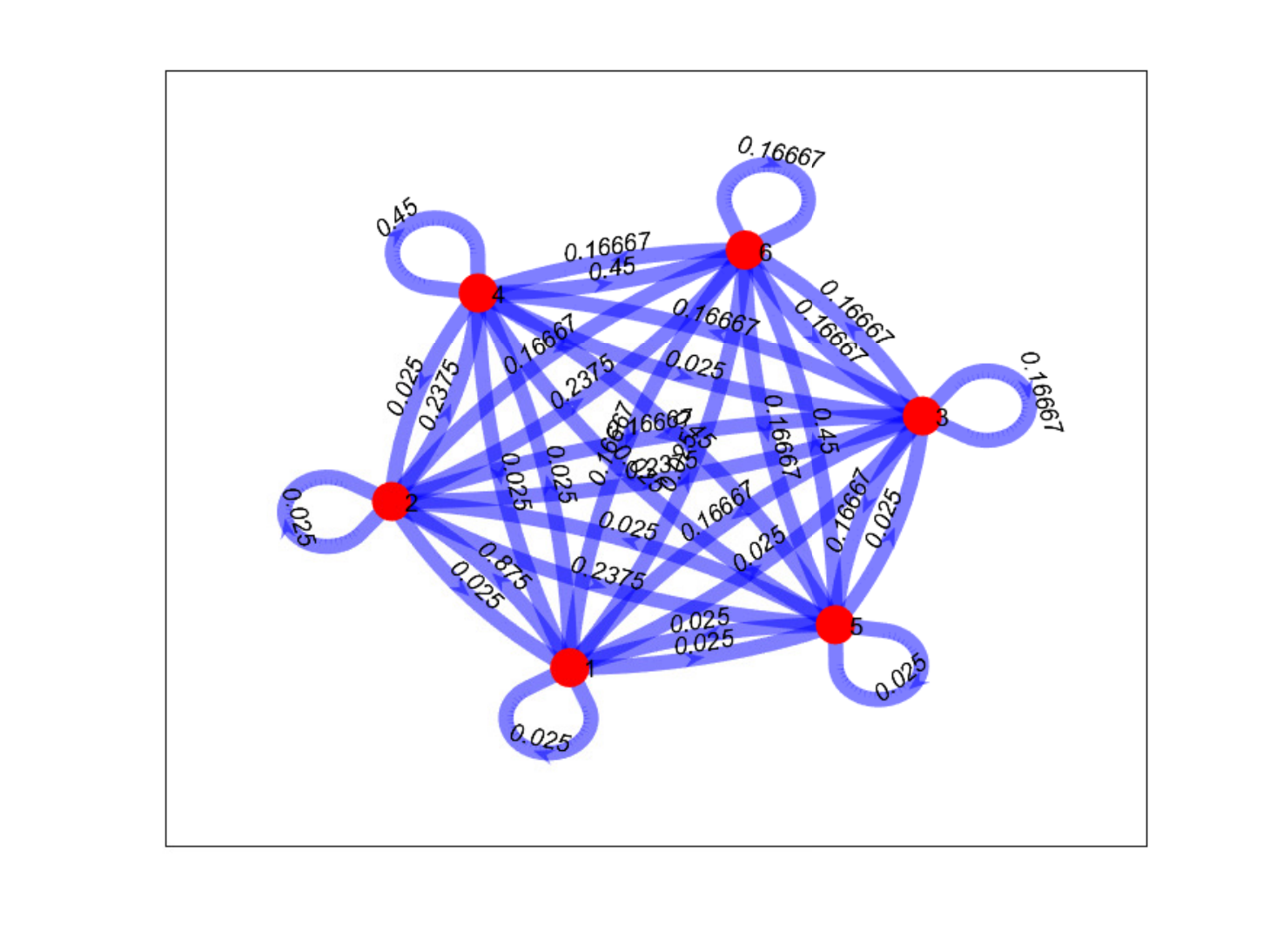}
\includegraphics[width=0.5\textwidth]{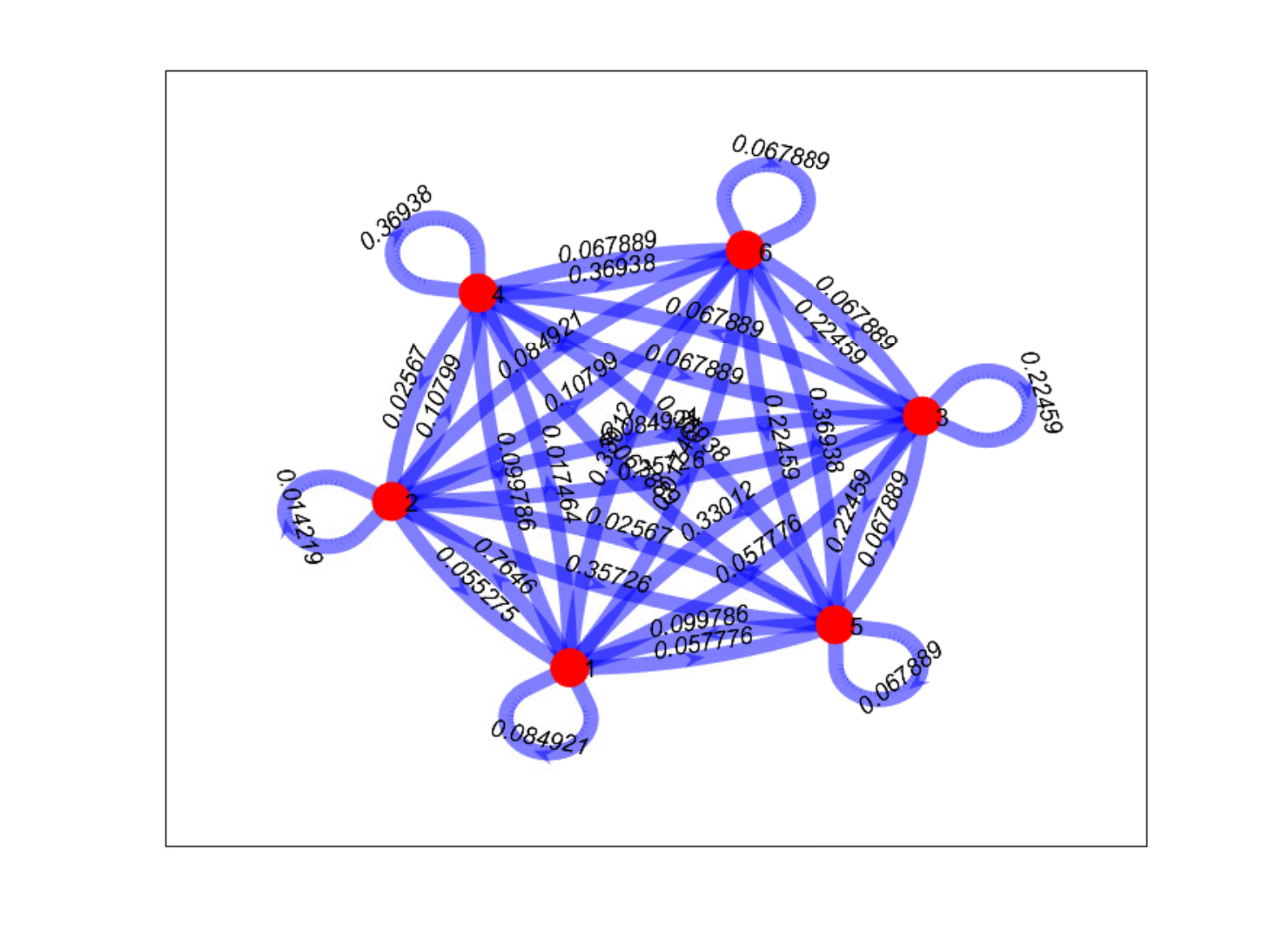}
\caption{The arced digraphs  corresponding to $\widetilde{C}$ (left) and $\widehat{C}$ (right) in Example \ref{ex:3}.} \label{fig51}
\end{figure}

We use  Algorithms \ref{nm}  and \ref{nm1} to Example \ref{ex:3}, where the initial guess $Z_0\in\cz$ is generated as in \eqref{sp-1} and the other parameters are set as above. The numerical results for Example \ref{ex:3} are listed in Table \ref{table3}. We see from Table \ref{table3} that both  Algorithms \ref{nm}  and \ref{nm1} can find a solution to the PDStIEP. The computed positive doubly stochastic matrix
by  Algorithm \ref{nm} is
\BE\label{sol-1}
C_*=
\left[
\begin{array}{cccccc} 0.1279  &  0.1447  &  0.0158  &  0.1858  &  0.1935  &  0.3323 \\
    0.2465  &  0.1898  &  0.1620  &  0.2230  &  0.1086  &  0.0700 \\
    0.2429  &  0.1012  &  0.0772  &  0.3368  &  0.0254  &  0.2166 \\
    0.1750  &  0.1778  &  0.0644  &  0.1330  &  0.2763  &  0.1735 \\
    0.0851  &  0.1670  &  0.4868  &  0.0538  &  0.1504  &  0.0569 \\
    0.1226  &  0.2195  &  0.1938  &  0.0677  &  0.2458  &  0.1506
\end{array}
\right]
\EE
with the real Schur form
\[
T=
\left[
\begin{array}{rrrrrr}
    1.0000 &  -0.0000  &  0.0000  &  0.0000 &  -0.0000 &   0.0000 \\
         0 &  -0.0856  &  0.4259  & -0.0281 &   0.1168 &  -0.0720 \\
         0 &  -0.2613  & -0.0856  &  0.1029 &   0.0358 &   0.0047 \\
         0 &        0  &       0  &  0.0000 &   0.0583 &   0.1008 \\
         0 &        0  &       0  &       0 &  -0.0000 &  -0.1268 \\
         0 &        0  &       0  &       0 &        0 &   0.0000
\end{array}
\right].
\]
The computed positive doubly stochastic matrix by  Algorithm \ref{nm1} is
\BE\label{sol-2}
C_*=
\left[
\begin{array}{cccccc}
    0.1267  &  0.1530  &  0.0161  &  0.1823  &  0.1895  &  0.3324  \\
    0.2502  &  0.1856  &  0.1574  &  0.2267  &  0.1086  &  0.0716  \\
    0.2512  &  0.0979  &  0.0809  &  0.3333  &  0.0256  &  0.2111  \\
    0.1706  &  0.1778  &  0.0665  &  0.1331  &  0.2784  &  0.1736  \\
    0.0832  &  0.1603  &  0.4858  &  0.0567  &  0.1527  &  0.0613  \\
    0.1182  &  0.2255  &  0.1933  &  0.0679  &  0.2451  &  0.1499
\end{array}
\right]
\EE
with the real Schur form
\[
T=
\left[
\begin{array}{rrrrrr}
    1.0000  &  0.0000 &  -0.0000  &  0.0000 &  -0.0000  &  0.0000 \\
         0  & -0.0856 &   0.4178  & -0.0262 &   0.1198  & -0.0830 \\
         0  & -0.2664 &  -0.0856  &  0.0937 &   0.0513  &  0.0012 \\
         0  &       0 &        0  &  0.0000 &   0.0666  &  0.0952 \\
         0  &       0 &        0  &       0 &  -0.0000  & -0.1296 \\
         0  &       0 &        0  &       0 &        0  &  0.0000
\end{array}
\right].
\]
The digraphs corresponding to the computed solutions are displayed in Figure \ref{fig52}.

Moreover, for the solution $C_*$ defined by \eqref{sol-1}, by using Algorithm  \ref{isc}, we can obtain the computed matrix
\[
\Theta=
\left[
\begin{array}{rrrrrr}
    0.4082 &   0.4510 &   0.2960  &  0.7346 &  -0.2122 &  -0.4807 \\
    0.4082 &  -0.0580 &  -0.3817  & -0.2308 &   0.7245 &  -0.5237 \\
    0.4082 &   0.3621 &  -0.6488  &  0.0503 &  -0.0782 &   0.2870 \\
    0.4082 &   0.1870 &   0.4218  & -0.6460 &  -0.2212 &   0.0423 \\
    0.4082 &  -0.7738 &  -0.0876  & -0.2025 &  -0.5637 &  -0.0109 \\
    0.4082 &  -0.1683 &   0.4002  &  0.2944 &   0.3507 &   0.6859
\end{array}
\right]\equiv[\bd\theta_1,\bd\theta_2,\bd\theta_3,\bd\theta_4,\bd\theta_5,\bd\theta_6],
\]
where $\mathcal{R}(\bd\theta_1)$,  $\mathcal{R}([\bd\theta_2,\bd\theta_3])$, and $\mathcal{R}([\bd\theta_4,\bd\theta_5,\bd\theta_6])$, respectively, form the invariant subspaces of $C_*$ corresponding to the eigenvalues of
\[
T_{11}=1.0000,\quad T_{22}=\left[
\begin{array}{rr}
-0.0856  &  0.4259   \\
-0.2613  & -0.0856
\end{array}
\right],\quad T_{33}=\left[
\begin{array}{rrr}
0.0000 &   0.0583 &   0.1008 \\
0 &  -0.0000 &  -0.1268 \\
0 &        0 &   0.0000
\end{array}
\right].
\]

Similarly, for the solution $C_*$ defined by \eqref{sol-2}, by using Algorithm  \ref{isc}, we can obtain the computed matrix
\[
\Theta=
\left[
\begin{array}{rrrrrr}
    0.4082 &   0.4383 &   0.3149 &   0.7214 &  -0.2044 &  -0.4622 \\
    0.4082 &  -0.0298 &  -0.3879 &  -0.2069 &   0.7223 &  -0.5384 \\
    0.4082 &   0.3819 &  -0.6306 &   0.0432 &  -0.0604 &   0.2945 \\
    0.4082 &   0.1726 &   0.4331 &  -0.6589 &  -0.2088 &   0.0382 \\
    0.4082 &  -0.7710 &  -0.1182 &  -0.1822 &  -0.5958 &  -0.0169 \\
    0.4082 &  -0.1920 &   0.3888 &   0.2835 &   0.3471 &   0.6848
\end{array}
\right]\equiv[\bd\theta_1,\bd\theta_2,\bd\theta_3,\bd\theta_4,\bd\theta_5,\bd\theta_6],
\]
where $\mathcal{R}(\bd\theta_1)$,  $\mathcal{R}([\bd\theta_2,\bd\theta_3])$, and $\mathcal{R}([\bd\theta_4,\bd\theta_5,\bd\theta_6])$, respectively, form the invariant subspaces of $C_*$ corresponding to the eigenvalues of
\[
T_{11}=1.0000,\quad T_{22}=\left[
\begin{array}{rr}
-0.0856 &   0.4178   \\
-0.2664 &  -0.0856
\end{array}
\right],\quad T_{33}=\left[
\begin{array}{rrr}
0.0000 &   0.0666  &  0.0952    \\
     0 &  -0.0000  & -0.1296  \\
     0 &        0  &  0.0000
\end{array}
\right].
\]
\begin{table}[ht]\renewcommand{\arraystretch}{1.0} \addtolength{\tabcolsep}{2pt}
  \caption{Numerical results for Example \ref{ex:2}.}\label{table3}
  \begin{center} {\scriptsize
   \begin{tabular}[c]{|c|r|r|r|r|l|l|}
     \hline
Alg. & {\tt CT.} & {\tt IT.} & {\tt NF.} &  {\tt NCG.} &  {\tt Res.}   &  {\tt grad.}  \\  \hline
Alg. \ref{nm} &  0.0150 s  & 6  & 7  & 50 & $1.29\times 10^{-11}$   & $1.38\times 10^{-11}$  \\
\cline{2-7}\hline
Alg. \ref{nm1} &  0.0160 s  & 7  & 8  & 53 & $6.54\times 10^{-13}$   & $5.83\times 10^{-13}$  \\
 \cline{2-7}\hline
  \end{tabular} }
  \end{center}
\end{table}
\begin{figure}[http]
\includegraphics[width=0.5\textwidth]{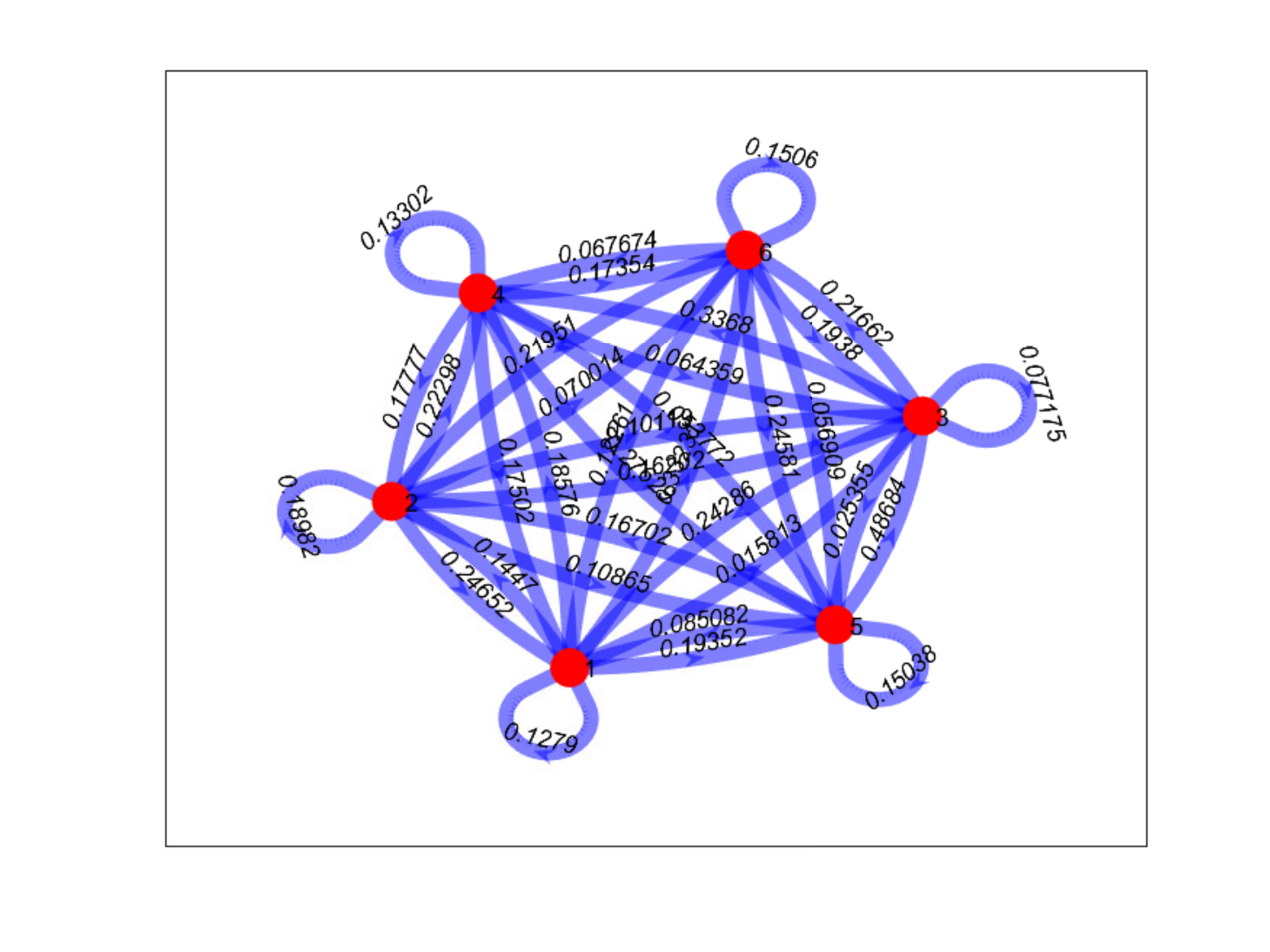}
\includegraphics[width=0.5\textwidth]{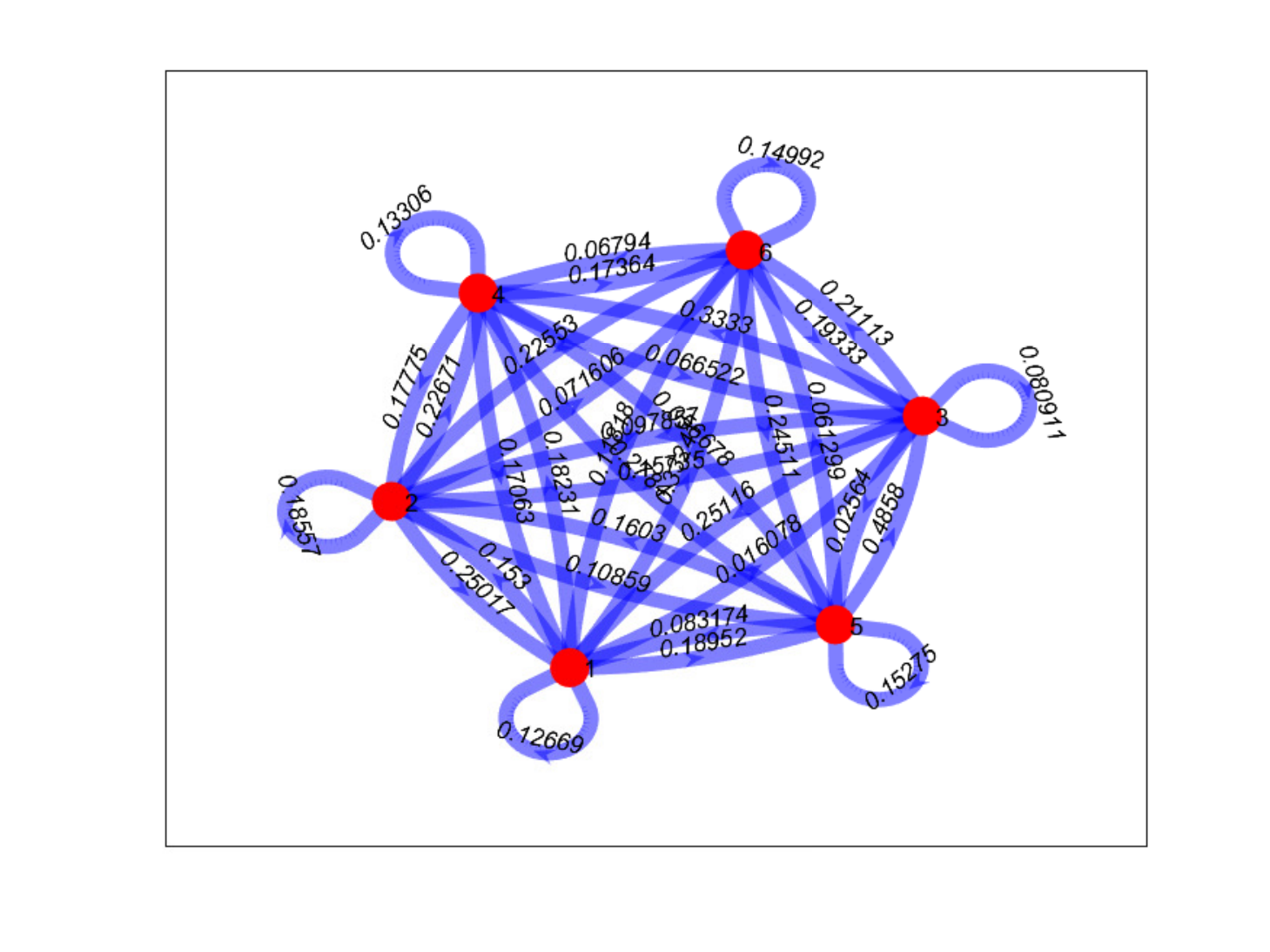}
\caption{The arced digraphs  corresponding to $C_*$ computed by Algorithm \ref{nm}   (left)  and Algorithm \ref{nm1} (right)  for Example \ref{ex:3}.} \label{fig52}
\end{figure}

\section{Concluding remarks}\label{sec:6}
In this paper, we present both monotone and nonmonotone Riemannian inexact Newton-CG methods for solving the inverse eigenvalue problem of constructing  a positive doubly  stochastic matrix from the prescribed realizable eigenvalues. We show that our methods converge globally and quadratically  under some assumptions.
We also provide invariant subspaces of the constructed solution to the inverse problem via its real Schur decomposition. Finally, we present some numerical tests (including an application in digraph) to  demonstrate  the efficiency of our methods. We must point out that the solutions computed by our methods are dependent on the starting points. In addition, an interesting question is how to design a  Riemannian method for finding  a low-rank positive doubly  stochastic matrix from the prescribed several nonzero eigenvalues. These questions need further study.



\end{document}